\documentclass[times,sort&compress,3p]{elsarticle}
\journal{Journal of Multivariate Analysis}
\usepackage[labelfont=bf]{caption}
 
\usepackage{enumerate}

\usepackage{mathtools}
\usepackage{enumitem}
\usepackage{bm}
\usepackage{amsmath,amsfonts,amssymb,amsthm,booktabs,color,epsfig,graphicx,hyperref,url}

\theoremstyle{plain}

\newtheorem{lemma}{Lemma}[section]

\theoremstyle{definition} 
\newtheorem{definition}{Definition}[section]

\newcommand\blfootnote[1]{%
	\begingroup
	\renewcommand\thefootnote{}\footnote{#1}%
	\addtocounter{footnote}{-1}%
	\endgroup
}

\newcommand{\R}{\mathbb{R}}

\newcommand{\eps}{\varepsilon}

\newcommand{\argmin}{\operatornamewithlimits{\arg\min}}

\newcommand{\E}{\op{\mathbb{E}}}

\newcommand{\p}{\mathbb{P}}

\makeatletter
\def\z@first#1#2{#1}
\def\z@second#1#2{#2}
\def\z@zp@selectchar#1#2{
	\IfStrEqCase{#2}{%
		{p}{#1{(}{)}}%
		{P}{#1{)}{(}}%
		{c}{#1{[}{]}}%
		{C}{#1{]}{[}}%
		{a}{#1{\{}{\}}}%
		{A}{#1{\}}{\{}}%
		{i}{#1{[}{]}\!#1{[}{]}}%
		{I}{#1{]}{[}\!#1{]}{[}}%
		{t}{#1{<}{>}}%
		{T}{#1{>}{<}}%
		{b}{#1{|}{|}}%
		{n}{#1{\|}{\|}}%
		{v}{#1{.}{.}}%
	}[#1{(}{)}]%
}
\def\z@zp#1#2\fin#3{
	\z@zp@selectchar{\left\z@first}{#1}#3
	\zifempty{#2}%
	{\z@zp@selectchar{\right\z@second}{#1}}%
	{\z@zp@selectchar{\right\z@second}{#2}}%
}
\newcommand{\zp}[2][p]{\zifempty{#1}{\left(#2\right)}{\z@zp#1\fin{#2}}}
\makeatother

\usepackage{xstring}
\def\zifempty#1#2#3{\def\foo{#1}\ifx\foo\empty\relax#2\else#3\fi}

\newcommand{\lam}{\lambda}

\newcommand{\vfi}{\varphi}
\newcommand{\al}{\alpha}

\newcommand{\bgt}{\begin{itemize}}
	\newcommand{\ent}{\end{itemize}}

\newcommand{\brem}{\begin{rmk}}
	\newcommand{\erem}{\end{rmk}}
\newcommand{\blem}{\begin{lem}}
	\newcommand{\elem}{\end{lem}}
\newcommand{\bcor}{\begin{cor}}
	\newcommand{\ecor}{\end{cor}}
\newcommand{\bTh}{\begin{Th}}
	\newcommand{\eTh}{\end{Th}}
\newcommand{\bpropo}{\begin{propo}}
	\newcommand{\epropo}{\end{propo}}

\newcommand{\op}{\operatorname}

\newcommand{\f}{\frac}
\newcommand{\ff}{\frac{1}}

\newcommand{\bbm}{\begin{bmatrix}}
	\newcommand{\ebm}{\end{bmatrix}}
\newcommand{\bes}{\begin{equation*}}
	\newcommand{\ees}{\end{equation*}}
\newcommand{\be}{\begin{equation}}
	\newcommand{\ee}{\end{equation}}
\newcommand{\beqy}{\begin{eqnarray}}
	\newcommand{\eeqy}{\end{eqnarray}}
\newcommand{\beq}{\begin{eqnarray*}}
	\newcommand{\eeq}{\end{eqnarray*}}

\newcommand{\bpm}{\begin{pmatrix}}
	\newcommand{\epm}{\end{pmatrix}}

\newtheorem{Th}{Theorem}[section]
\newtheorem{propo}[Th]{Proposition}
\newtheorem{Prop}{Proposition}[section] 
\newtheorem{lem}[Th]{Lemma}

\newtheorem{cor}[Th]{Corollary}

\newtheorem{rmk}[Th]{Remark}
\newtheorem{Def}[Th]{Definition}

\theoremstyle{definition}

\long\def\symbolfootnote[#1]#2{\begingroup
	\def\thefootnote{\fnsymbol{footnote}}\footnote[#1]{#2}\endgroup}

\makeatletter
\def\@addpunct#1{%
	\relax\ifhmode
	\ifnum\spacefactor>\@m \else#1\fi
	\fi}

\makeatletter
\def\blfootnote{\gdef\@thefnmark{}\@footnotetext}
 
\makeatletter
\def\ps@pprintTitle{%
  \let\@oddhead\@empty
  \let\@evenhead\@empty
  \let\@oddfoot\@empty
  \let\@evenfoot\@oddfoot
}
\makeatother

\begin{document}

\begin{frontmatter}
 
\title{Estimation of extreme $L^1$-multivariate expectiles with functional covariates}

\author[1]{Elena Di Bernardino}
\author[1]{Thomas Lalo\"{e}}
\author[1]{Cambyse Pakzad\corref{mycorrespondingauthor}}
\address[1]{Université Côte d’Azur, Laboratoire J.A. Dieudonné, UMR CNRS 7351,   Nice,  France.}
  
\cortext[mycorrespondingauthor]{Corresponding author. Email address: \url{cpakzad@unice.fr}}

\begin{abstract} 
The present article is devoted to the semi-parametric estimation of   multivariate  expectiles  for extreme levels. The considered multivariate  risk measures also include the possible conditioning with respect to a  functional covariate, belonging to an infinite-dimensional space.  By using the first order 
 optimality condition, we interpret these expectiles as solutions of a multidimensional nonlinear optimum problem. Then the inference is based on a minimization algorithm of gradient descent type, coupled with consistent kernel estimations of our key statistical quantities such as conditional quantiles, conditional tail index and conditional tail dependence functions. The method is valid for equivalently heavy-tailed marginals and under a multivariate regular variation condition on the underlying unknown  random vector with  arbitrary dependence structure. Our main result establishes the consistency in probability of the optimum approximated solution vectors with a speed rate. This allows us to estimate the global computational cost of the whole procedure according to the data sample size. 
\end{abstract}


\end{frontmatter}

\section{Introduction\label{sec:1}}
 
Risk measurement theory is an active branch of research with numerous applications in the fields of finance, insurance, economics and for the environment (hydrology, geology, \ldots). 

To study the extreme risk of a random phenomena, e.g.\ of a big loss on a financial position, the most popular way is to estimate quantiles, also known as Value-at-Risk (VaR), at high level. But it has been argued that it lacks of sub-additivity property \cite{acerbi}, i.e.\ the quantile of the sum of two portfolios can exceed the sum of the quantiles of the two portfolios; which is contradiction with the principle of diversification. Consequently, the quantile is an incoherent risk measure in view of \cite{Artzner}. Another drawback is that it relies on the frequency of tail events and not on their real magnitudes which is precisely what one would like to know. On the other hand, the second most famous risk measure, the expected shortfall, is not elicitable in the sense of \cite{Gneiting2011}, meaning that it is not defined as the minimization of the expectation of some score function. This is though a desirable property since it allows backtesting procedure: one periodically compares the expected risk measure with the actual value of the variable of interest in order to evaluate the accuracy of the forecasting methodology, see \cite{Bignozzi}. The univariate expectiles are then introduced by \cite{NeweyPowell1987} and turn out to be the only law invariant risk measures which are elicitable by construction and coherent for a threshold level range, see \cite{Bellini:coherent,Ziegel}. Concerning the univariate expectile, we also refer to \cite{Bellini, daouia:extreme_expectile,daouia:extreme_expectile2}.  Conversely to quantiles, they depend on both the realisations and probabilities of the underlying random variable. In economic terms, they may be interpreted as ratios of expected gain/loss which found to be recognized in portfolio management, see \cite{Bellini}, and represent the quantity of money to inject in a position to reach a prescribed ratio gain/loss. 

A fundamental question for potential practice is the statistical estimation in the extreme regime of the risk measures at our disposal. In this context, the extreme regime is modelled by a risk level tending towards zero or one, and an assumption of heavy-tails on the underlying distributions, typically of Pareto-type, which best captures rare phenomena. The proper mathematical framework is the regular variation  notion.   Besides, it often happens in practical applications that the observations are recorded along with auxiliary information represented by a random covariate. Then, one would preferably take advantage of the extra information by focusing on the conditional extremes. This line of research was carried by \cite{GSU_non_funct,GSU_AoS} when the covariate is a random vector. A valuable improvement is \cite{GSU2022} for univariate extreme expectiles when the covariate belongs to an infinite-dimensional or functional space.

Furthermore, in several situations, one needs to simultaneously manage risks over different positions, requiring a multivariate version of risk measures which would take the underlying dependence structure into account. Recently, several possible multivariate extensions of expectiles emerged in the literature such as geometric expectiles in \cite{herrmann} and $L^p$-expectiles in \cite{maumedeschamps:extension}. In the present work, we focus on the specific case of the multivariate  $L^1$-expectiles of \cite{maumedeschamps:extension} in the extreme regime, which we abbreviate with the notation MEEs for \emph{Multivariate Extreme Expectiles}. Herein, in \cite{maumedeschamps:extension} the authors construct $\Sigma$-expectiles as another extension possibility based on correlation matrices $\Sigma$ that reduces to the $L^1$-version if $\Sigma\equiv \bm{1}$. To estimate multivariate expectiles, they exploit  in \cite{maumedeschamps:extension} the elicitibality by focusing on the first order optimality condition, namely MEEs are points in the $d$-dimensional Euclidean space for which the score gradient vanishes. The latter only involving tail expectations through positive and negative parts, the authors achieve the estimation by means of Robbins-Monro’s stochastic approximation method for moderate levels of risk. Conversely, in the extreme regime, the same authors use in \cite{MaumeDeschampsRulliereSaidExtremes} classic regular variation tools to express MEEs ratios as solutions of a system of coupled nonlinear equations.   In addition, an estimation procedure is given for $L^1$-MEEs with equivalent regularly varying marginal tails, provided the tail dependence is either comonotonic or asymptotically independent. Then, the  approach  in \cite{MaumeDeschampsRulliereSaidExtremes}   works well for specific dependence structures.\\ 

In this paper, we address the estimation of the $L^1$-MEEs when a conditional  covariate lying in a possibly infinite-dimensional space is available.  We assume the equivalent regularly varying marginal tails hypothesis and that the underlying dependence structure and the marginal distributions are unknown.  Again, the first order optimality condition yields a system of equations for which functional  MEEs ratios are solutions. Equivalently, functional  MEEs may be seen as roots of a certain loss function which can be turned into an optimization problem. Crucial quantities are involved regarding the tail behaviour such as the conditional tail index and the  conditional tail dependence function. Inspired by \cite{BeckMailhotElena2021}, plugging their empirical counterparts in the optimization problem results in the approximated  loss function. Finally,  we propose to apply a BFGS-gradient method. In our main result (see Theorem \ref{th:main2}), we prove the consistency in probability with  rate of the approximated loss function and of the associated optimum solution in this conditional functional setting. 
Contrarily to \cite{BeckMailhotElena2021}, in the present work,  a special attention is devoted to the rates of convergence in the approximation of the underlying optimization problem. It allows to link the loss function approximation quality with the steps of the used  gradient descent algorithm. As a result, we explicitly provide the speed rate at which the approximated optimum converges to its theoretical value, which essentially quantifies the estimation quality according to the sample size.\\


The paper is organized as follows.  
In Section \ref{sec:preliminaries}, we present some necessary notation   and the setting of our model.  
We introduce in Section \ref{Funcional MEEs} the formal definition of functional multivariate $L^1$-expectiles and {how they can be theoretically related to an non-linear optimum problem.}
We subsequently develop in Section \ref{StatTOOLS} the statistical tools of the present  paper. After introducing the different estimators, we construct the associated approximated optimum problem. 
In Section \ref{convergence} we present the required hypotheses  and we state our main result (see Theorem \ref{th:main2}) about the convergence with rate of the approximated optimum problem to the theoretical one. We outline the sketch of the proof of our main result and we give a crucial intermediate convergence result on the loss function and its gradient is also given (see Proposition  \ref{prop:main}). We devote the last part of Section \ref{convergence}  to a discussion about the required hypotheses. 
Section \ref{sectionProof}  is devoted to the proof of the main result.  Auxiliary proofs and supplementary lemmas are postponed to  Section \ref{annexeA}.  Further material about the   second order regular variation condition is provided in   Appendix \ref{sec:RVF}. 
 
\section{Notation and preliminaries} \label{sec:preliminaries}

In this work, we consider a Polish space $(E,\zp[n]{\cdot}_E)$ endowed with its Borel $\sigma$-algebra and a probability space $(\Omega,\mathcal{A},\p)$. Let $2\le d <+\infty$ and a random pair $\left( \bm X,Y\right)\in \R^d\times E $, with $\bm X = (X_j)_{1\le j \le d}$, defined on $\Omega$ such that $\bm X \in (L^1(\Omega))^d$. On account of \cite[Theorem 3.2]{kallenberg2002foundations}, such topological features on $E$ ensures the existence of regular conditional probabilities $\p_y(\cdot) := \p \left(\cdot|Y=y \right)$ on $\mathcal{A}$ for $\p_Y$-almost all $y\in E$, where $\p_Y=\p \circ Y^{-1}$ is the pushforward measure. As a result, we may denote the conditional cumulative distribution function (cdf) of $\bm X$ given $Y=y$ by $F_{\bm X,y}(\bm x) := \p_y(\{ \bm X\le \bm x\})= \p \left( \bm X \le \bm x | Y=y \right)$,  for  $\bm x \in \R^d$ and $y\in E$, where the inequality is to be understood component-wise on $\R^d$.
  
We denote the conditional marginal distributions, which we suppose to be continuous, by $$x_j\mapsto F_{j,y}(x_j):=F_{\bm X,y}(+\infty,\ldots,x_j,+\infty,\ldots),\quad 1\le j\le d,$$
and for $1\le j< k\le d$,  the conditional bivariate marginal distributions by
\begin{equation}\label{Fbivariate}
(x_j,x_k)\mapsto F_{j,k,y}(x_j,x_k):=F_{\bm X,y}(+\infty,\ldots,x_j,+\infty,\ldots,x_k,+\infty,\ldots).
\end{equation}

We also introduce the conditional marginal quantile function of $X_j$ given $Y=y$ and its conditional marginal quantile tail function by,  respectively,
\begin{align}
 q_{j,y}(\al)&:=\inf \left\{x>0,{F}_{j,y}(x)\ge \al \right\},\quad \al \in (0,1), \label{marginalquantile}\\
  U_{j,y}(x) &:= q_{j,y}(1-x^{-1})= \inf \left\{ t\in \R, F_{j,y}(t)\ge 1-x^{-1} \right\} , \quad x>0. \nonumber
\end{align}
Under the hypothesis of continuity of conditional marginal distributions   and by Sklar's Theorem (1959) (see \cite{Sklar}), there exists an unique copula $C_y$ such that, for $\bm u = (u_1,\ldots,u_d)\in [0,1]^d$,
\begin{align*}
	C_y\left( \bm u \right) &= F_{\bm X,y}\left( F^{-1}_{1,y}(u_1),\ldots,F^{-1}_{d,y}(u_d)  \right)
	=\p\Big(\bigcap_{j=1}^d\left\{F_{j,y}(X_{j})\le u_j  \right\}\big| Y=y \Big),
\end{align*}
for which we assume the regularity condition from \cite{irene,Segers2011}; namely, for each $j=1,\ldots,d$,
\begin{align}\label{hyp:segers}
	\bm u\mapsto	\f{\partial C_y}{\partial u_j}(\bm u) \text{ exists and is continuous on the set $\{ \bm u \in [0,1]^d:~0<u_j<1 \}$. }
\end{align}

For $y\in E$ and $ \bm x\in \R_{+}^d$, we define the conditional stable tail dependence function
\begin{align}\label{eq:ell_def_cond}L_y(\bm x)&:= \lim\limits_{t\downarrow 0}t^{-1} \p\Big(\bigcup_{j=1}^d\left\{1-F_{j,y}(X_{j})\le t\,x_j  \right\}\big| Y=y\Big)=  \lim\limits_{t\downarrow 0}t^{-1} \left( 1-C_{y}(\bm 1-t\,\bm x)\right). \end{align}
 
For  any $y\in E$  and $\bm x \in [0,+\infty]^d \setminus \{(+\infty,\ldots,+\infty)\}$, the associated conditional upper tail dependence function is given by  \begin{align}\label{eq:lam_def_cond} \lam_y(\bm x)&= \lim\limits_{t\downarrow 0}t^{-1} \p\Big(\bigcap_{j=1}^d\left\{1-F_{j,y}(X_{j})\le tx_j  \right\}\big| Y=y\Big)=   \zp[n]{\bm x}_1-L_y(\bm x).
\end{align} 

Let $1\le j<k\le d$, $\bm u = (u_1,\ldots,u_d)\in [0,1]^d$ and $\bm u^{j,k} = \left(\bm 1_{\ell\notin \{j,k\}}+ u_\ell \bm 1_{\ell\in \{j,k\}}\right)_{1\le \ell\le d} $.  By using \eqref{eq:lam_def_cond} and for $ \bm x\in \R^2_+$, we introduce the bivariate restrictions of the conditional copula and tail dependence functions as following

\begin{align*}C_{j,k,y}\left( u_j,u_k \right)&:=C_y\left( \bm u^{j,k} \right) = \p\left(F_{j,y}(X_{j})\le u_j,F_{k,y}(X_{k})\le u_k   \big| Y=y \right),\\
L_{j,k,y}(\bm x)&:= \lim\limits_{t\downarrow 0}t^{-1} \left( 1-C_{j,k,y}(\bm 1-t\bm x)\right), \quad \lam_{j,k,y}(\bm x) = \zp[n]{\bm x}_1-L_{j,k,y}(\bm x)= \lim\limits_{t\downarrow 0}t^{-1}{\overline{C}_{j,k,y} (t\bm x)},
\end{align*}
where $\overline{C}_{j,k,y}$ is the survival copula  defined by $\overline{C}_{j,k,y}(u,v)=u+v-1+C_{j,k,y}(1-u,1-v)$ on $[0,1]^2$. In the subsequent paper, we suppose that each bivariate tail dependence function is continuous. Furthermore, we quantify the convergence rate in \eqref{eq:ell_def_cond}, akin to \cite{Einmahl_stdf,Schmidt2006}, by assuming that for any $1\le j\neq k\le d$, there exists $\mu_{j,k,y}>0$ such that, as $t\downarrow 0$, and for any $T>0$, 
\begin{align*}\sup_{\bm x \in [0,T]^d}\zp[b]{t^{-1} \left( 1-C_{j,k,y}\left(\bm 1 - t\bm x \right) \right)-L_{j,k,y}(\bm x)}& = O(t^{\,\mu_{j,k,y}}).
\end{align*}

In particular, defining $\mu_y :=\max\limits_{1\le j\neq k\le d}\{\mu_{j,k,y}\} $, we may write 
\begin{align}\label{hyp:rate_theoretical}\max_{1\le j\neq k\le d}\sup_{\bm x \in [0,T]^d}\zp[b]{t^{-1} \left( 1-C_{j,k,y}\left(\bm 1 - t\bm x \right) \right)-L_{j,k,y}(\bm x)}& = O\left( t^{\,\mu_y}\right) .\end{align}

We now state  hypothesis on the considered random vector $\bm X$ given $Y=y$: it has equivalent regularly varying marginal  tails. More precisely, as in  \cite{BeckMailhotElena2021,MaumeDeschampsRulliereSaidExtremes},  we assume that each conditional marginal function  behaves the same way in the extreme regime.

\begin{definition}[Considered conditional marginal tails model]\label{hyp:model}\textit{
For any $1\le j \le d$ and $y\in E$, there exists $\gamma_y>0$, $\rho_{j,y}\le 0$, $c_{j,y}>0$, such that
	\begin{enumerate}[label=($\mathcal{H}$\arabic*)]
\item \label{hyp:model1} $\overline{F}_{j,y}(\cdot)=\p\left(X_j> \cdot | Y=y\right)\in  2{\rm{RV}}_{-\ff{\gamma_y},\f{\rho_{j,y}}{\gamma_y}}(+\infty)$,
		\item  \label{hyp:model2}$\lim\limits_{x\to \infty}\displaystyle{\f{\overline{F}_{j,y}(x)}{\overline{F}_{1,y}(x)} = c_{j,y}}<+\infty$.
	\end{enumerate}}
\end{definition}

Assumptions \ref{hyp:model1} and  \ref{hyp:model2} are classical in the extreme value literature. For a discussion about Assumptions \ref{hyp:model1} and  \ref{hyp:model2} the reader is referred to   Appendix \ref{sec:RVF}.

\section{Functional MEEs and optimum system}\label{Funcional MEEs}

The multivariate risk measure studied in the present  paper is introduced in the following definition.    
\begin{definition}[Multivariate $L^1$-expectile with  functional covariate]\label{defExpectile}
\textit{Let $y\in E$. We define the  multivariate $L^1$-expectile with  functional covariate $Y$ at risk level $\al \in (0,1)$ as the random vector in $\R^d$, such that
\begin{align}\label{def:cond_multivariate_L1-expectile1}
	\bm{e}_\al( \bm X,y) &:= \argmin_{\bm x\in \R^d} \E\left[\al \zp[n]{(\bm X-\bm x)_+ }^2_1 +(1-\al) \zp[n]{(\bm X-\bm x)_- }^2_1 |Y=y\right]
	\\&=   \argmin_{{(x_1,\ldots,x_d)\,\in \R^d}} \E\left[\al \Big(\sum_{j=1}^d \left(X_j-x_j\right)_+ \Big)^{2 }+(1-\al)\Big(\sum_{j=1}^d \left(X_j-x_j\right)_- \Big)^{2 }|Y=y\right].\nonumber
\end{align}}
 \end{definition}

Definition \ref{defExpectile}  can be seen as a  functional conditional extension  of the multivariate $L^1$-expectiles recently introduced in  \cite{BeckMailhotElena2021} (for further  details the interested reader is referred to Equation (2) and Conclusion section  in \cite{BeckMailhotElena2021}).  Observe that the equality  in \eqref{def:cond_multivariate_L1-expectile1} is assured by the strict convexity of $s_\al(t,x)=\al \zp[n]{(t-x)_+ }_{1}^2 +(1-\al) \zp[n]{(t-x)_- }_{1}^2$ in the variable $x$ which transfers to $\E \left(s_\al(X,x)\right)$ as well by Jensen inequality. Obviously, the conditional $L^1$-expectile in \eqref{def:cond_multivariate_L1-expectile1} is a random vector in $\R^d$ and we write {$ \bm{e}_\al( \bm X,y)=(\bm{e}^j_\al( \bm X,y))_{1\le j\le d}.$}  Furthermore, concerning the endpoint, one has $x_F=\sup\{x,F(x)<1\}=+\infty$, so that $\bm{e}^j_\al( \bm X,y) \rightarrow \infty$, for $\al\to 1$ and  $1\le j \le d.$\\

Definition in \eqref{def:cond_multivariate_L1-expectile1} can be reformulated in terms of first order optimality condition, giving rise to a system in $\R^d$. As suggested in \cite{BeckMailhotElena2021}, we quickly convert the work of \cite{MaumeDeschampsRulliereSaidExtremes} to the functional setting to provide an expression of the optimum system. This involves other quantities such as the conditional tail index in the heavy-tailed case and the conditional bivariate tail dependence function. One may without difficulty transpose Propositions 3.1, 3.3 and 5.1 in \cite{MaumeDeschampsRulliereSaidExtremes} to our functional setup. In particular, suppose that the random vector $\bm X$ given $Y=y$ is regularly varying (MRV)  (see Appendix  \ref{sec:RVF}) of index $1/\gamma_y$ with conditional marginal tails satisfying  Conditions \ref{hyp:model1} and \ref{hyp:model2}. Then, as $\alpha \rightarrow 1$, any limiting vector $(\eta,\beta_2,\ldots,\beta_d)$ of $\left( \frac{1-\al}{\overline{F}_{1,y}\left(\bm{e}_\al^1(\bm X,y)\right)},\frac{\bm{e}_\al^2(\bm X,y)}{\bm{e}_\al^1(\bm X,y)} ,\ldots, \frac{\bm{e}_\al^d(\bm X,y)}{\bm{e}_\al^1(\bm X,y)}\right),$ for $y\in E,$ satisfies the system with $k=1,\ldots,d$,
	\begin{align}\label{eq:system_CEME}
		\ff{1/\gamma_y-1}-\eta \frac{\beta_k^{1/\gamma_y}}{c_{k,y}}&=-\sum_{j\neq k}\left(  \int_{\frac{\beta_j}{\beta_k}}^{+\infty}\lam_{j,k,y}\left(\frac{c_{j,y}}{c_{k,y}}t^{-1/\gamma_y},1\right) \mathrm{d}t- \eta \frac{\beta_k^{1/\gamma_y-1}}{c_{k,y}}\beta_j\right).
	\end{align}
 
 In the case where $\bm \Theta :=(\eta,\beta_2,\ldots,\beta_d)$ is unique,   by using the quantile in \eqref{marginalquantile},  as $\al \to 1$,  we have  $$\bm{e}_\al( \bm X,y)\sim q_{1,y}(\al)\eta^{\gamma_y}\left( 1,\beta_2,\ldots,\beta_d\right)^t. $$

 For a comprehensive review of the different definitions available for the MRV property, we refer to \cite{resnick2007heavy}. This assumption on $\bm X| Y $ mainly ensures the existence of {the conditional} bivariate upper tail dependence functions which we suppose continuous, see Section 2.3 in \cite{MaumeDeschampsRulliereSaidExtremes}. Notice the system \eqref{eq:system_CEME} only displays pairwise interactions through the conditional bivariate tail dependence function. In view of \eqref{eq:system_CEME}, we introduce the associated optimum problem.

\begin{definition}[Loss function for the  optimum problem of   $\bm{e}_\al(\bm X,y)$] \label{defLoss}
\textit{Let   $y\in E$. We consider $\bm \Theta := (\eta,\beta_2,\ldots,\beta_d)>0$ as \eqref{eq:system_CEME}.  Consider also
\begin{align}\label{xivector}\bm \xi_y &= (\gamma_y,c_{2,y},\ldots,c_{d,y},\lam_y)\end{align}
 with $\lam_y$  as in \eqref{eq:lam_def_cond}, $\gamma_y$ as in Condition \ref{hyp:model1} and   $c_{j,y}$ as in Condition  \ref{hyp:model2}.   Define the vector $\bm{\vfi}_y \left( \bm \Theta ,\bm \xi_y \right)  = \left( \vfi_{k,y} ( \bm \Theta ,\bm \xi_y ) \right) _{1\le k\le d}$ 
 where
\begin{align}\label{eq:vfi}
	{\vfi}_{k,y} ( \bm \Theta ,\bm \xi_y ) &:=
	\frac{\gamma_y}{\gamma_y-1}-\eta \frac{\beta^{1/{\gamma_y}}_k}{c_{k,y}}\left(1+\sum_{j\neq k} \frac{\beta_j}{\beta_k} \right) +\sum_{j\neq k} \int_{\frac{\beta_j}{\beta_k}}^{+\infty}\lam_{ j,k,y}\left(\frac{c_{j,y}}{c_{k,y}}t^{-1/{\gamma_y}},1\right) \mathrm{d}t.
\end{align}
With these notations, one can formulate the system in \eqref{eq:system_CEME} as $\bm{\vfi}_y \left( \bm \Theta ,\bm \xi_y \right) = \bm 0$ (the equality being understood in $\R^d$).  We introduce the loss function via the squared Euclidean norm
 \begin{align}\label{eq:loss}
\mathcal{L}_{\bm \xi_y} \left( \bm \Theta  \right) &:= \ff{2} \zp[n]{\bm{\vfi}_y \left( \bm \Theta ,\bm \xi_y \right)}^2_2=\ff{2}\sum_{k=1}^d \vfi_{k,y}\left( \bm \Theta ,\bm \xi_y \right)^2,
\end{align} and an optimal vector $\bm \Theta^\star_y$ obtained by minimizing the loss function $\mathcal{L}_{\bm \xi_y}$ in \eqref{eq:loss}, \begin{align}\label{eq:optimization}\bm \Theta^\star_y\in &\argmin_{\bm \Theta} \mathcal{L}_{\bm \xi_y} ( \bm \Theta ).
\end{align} }
\end{definition}

To minimize the optimum problem \eqref{eq:optimization}, we turn to the Broyden–Fletcher–Goldfarb–Shanno (BFGS) algorithm which belongs to the the family of quasi-Newton optimization methods. This choice is motivated by practicality while, from a rigorous point of view, its global convergence is denied due to the loss function lacking of convexity property. A particular interest comes from the fact that no second derivatives computations are needed. In fact, we will rather consider the L-BFGS-B method, a refinement of the classic BFGS that incorporates bound constraints.

Since the underlying distribution is not known in practice, direct application of any minimization algorithm on the optimization problem \eqref{eq:loss}-\eqref{eq:optimization}  is not feasible.  The next section is devoted to overcome this drawback.

\section{Approximated optimum problem}\label{StatTOOLS}

Let $\left( \bm{X}_i, Y_i\right)$ for ${1\le i \le n}$,   with $\bm X_i=(X_{ij})_{ 1\le j \le d}$,  be  $n$ independent copies of $\left( \bm X,Y\right)\in \R^d\times E $. In the following the convergence in probability and in distribution, as $n$ goes to infinity, are respectively denoted by $\xrightarrow{\p}$\, and \, $\xrightarrow{d}$.  Relying on this data sample, we consider  an approximated version of the loss function for the  optimum problem of   $\bm{e}_\al(\bm X,y)$ previously introduced  in Definition \ref{defLoss}.

\begin{definition}[Approximated loss function of $\bm{e}_\al(\bm X,y)$] \label{defLossApprox}
\textit{By using \eqref{eq:loss} the  approximated loss function can be written as
$\mathcal{L}_{\bm{\hat \xi}_{n,y}} \left( \bm \Theta  \right) = \ff{2} \zp[n]{\bm{\vfi}_y \left( \bm \Theta , \bm{\hat \xi}_{n,y} \right)}^2_2$, where   $\bm{\hat \xi}_{n,y}$ is  an estimator vector of $\bm \xi_y$ in \eqref{xivector}.  Furthermore, we define the  associated    optimum problem  as
\begin{align}\label{eq:optimization_approx}
\hat{ \bm \Theta}^\star_{y}\in &\argmin_{\bm \Theta} \mathcal{L}_{\hat{ \bm \xi}_{n,y}} (\bm \Theta).
\end{align}}
\end{definition} 
We now  aim to consistently  estimate the  loss function { $\mathcal{L}_{\bm{  \xi}_{y}} \left( \bm \Theta  \right)$.}   To this purpose, we firstly  build a consistent estimator vector  $\bm{\hat \xi}_{n,y} = (\hat \gamma_{n,y},\hat c_{2,n,y},\ldots,\hat c_{d,n,y},\hat \lam_y)$ of $\bm \xi_y$.   According to \eqref{eq:system_CEME}, we  need to  estimate  $\lam_{j,k,y}$ (see Section \ref{subsection1Estimators} below) and  $\gamma_y$, $\bm{c}_y=\left(c_{2,y},\ldots,c_{d,y}\right)$ (see Section \ref{subsection2Estimators}).    The interested reader is referred to the framework in \cite{irene, GSU2022} (see also~\cite{GG2012}).

\subsection{Empirical counterpart for $\lam_{j,k,y}$}\label{subsection1Estimators}

 Let $1\le j<k\le d$.   We consider the empirical counterpart of the bivariate conditional cdf $F_{j,k,y}$ in  \eqref{Fbivariate} (see, \emph{e.g.}, \cite{Daouia2013,Ferraty2010,GSU2022})
\begin{align}\label{eq:cdf_estimator}
	\hat{F}_{j,k,n,y}\left( x_j,x_k\right) &:=\sum_{i=1}^n w_{i,n}\bm{1}_{\{ X_{ij}\le x_j,X_{ik}\le x_k \}},\, \, \mbox{ for } \,\, x_j,x_k\in \R,
\end{align}
where the Nadaraya-Watson weights are given by
\begin{align}\label{eq:weight} w_{i,n}&:=  w_{i,n}(Y,y,h_n)= \frac{K\left( \f{{\zp[n]{Y_i-y}_E} }{h_n}\right) }{\sum_{s=1}^nK\left( \f{{\zp[n]{Y_s-y}_E} }{h_n}\right)}, \, \, \mbox{ with } \,\, \sum_{i=1}^n w_{i,n} = 1,
\end{align}with  $K:\R\to \R^+$ a kernel function, \emph{i.e.} positive and measurable, having support in $[0,1]$.


We define the empirical estimator of the conditional copula and its bivariate restriction
\begin{align}\label{eq:copula_estimator}
\hat C_{n,y}(\bm u) &:= \sum_{i=1}^n w_{i,n}\bm{1}_{\{ \bigcap_{j=1}^d\{ \hat{F}_{j,n,y}( X_{ij})\le u_j \}\}},\quad
\hat{C}^{\, }_{j,k,n,y}(u_j,u_k):=\hat{C}_{n,y}(\bm u^{j,k}) =\sum_{i=1}^n w_{i,n}\bm{1}_{\{ \hat{F}_{j,n,y}( X_{ij})\le u_j, \hat{F}_{k,n,y}(X_{ik})\le u_k \}},
\end{align}
where $\hat{F}_{j,n,y}$, $\hat{F}_{k,n,y}$ are the corresponding marginal distribution functions of $\hat{F}_{j,k,n,y}$ in \eqref{eq:cdf_estimator}.
Furthermore, for $ \bm x\in \R^2_+$ and $k_n$ a deterministic intermediate sequence, i.e. $1\ll k_n \ll n$, we  consider the  empirical estimators of the bivariate conditional stable tail dependence function and of the associated upper  tail dependence function:
 \begin{align*}
 \hat L_{j,k,n,y}(\bm x) &:= \frac{n}{k_n}\Big( 1 - \hat C_{j,k,n,y}(\bm 1- \frac{k_n}{n}\bm x)  \Big),\quad
 \hat \lam_{j,k,n,y}(\bm x) = \zp[n]{\bm x}_1-\hat L_{j,k,n,y}(\bm x).
 \end{align*}

\subsection{Empirical counterparts for $\gamma_y$ and $\textbf{{c}}_y$}\label{subsection2Estimators}


Let $y\in E$. Assume  $h_n\ll 1$ be  a non-random positive sequence such that  the small ball probability function is positive, \emph{i.e.},  \begin{align}\label{smallBall}
\psi_y(h_n)&:= \p\left( {\zp[n]{Y-y}_E} \le h_n\right) >0.
\end{align}
 The latter quantity has a major impact in the functional estimation framework as one can see for instance in our   hypotheses in Section  \ref{AssumptionsMainTheorem} below. Note that it heavily relies on the intrinsic metric on $E$. \newline

Our plug-in estimation relies on functional estimators of the marginal quantile and tail index previously introduced by  \cite{GG2012,GSU2022}. Let us then define
\begin{align*}
\hat{q}_{j,y,n}(\al)&:=\inf \left\{x>0,\hat{{F}}_{j,y,n}(x)\ge  \al \right\},\quad \al \in [0,1],
\end{align*}
the natural empirical counterpart of the conditional quantile function in \eqref{marginalquantile}, where $\hat{{F}}_{j,y,n}$, is  the  marginal distribution function associated to \eqref{eq:cdf_estimator}.

We pick an integer $1 \le J<+\infty$ and a subdivision  $0<\tau_J<\tau_{J-1}<\cdots< \tau_i< \cdots \tau_2<\tau_1\le 1$, \emph{e.g.} $\tau_i=1/i$ and $J=9$ are motivated by the discussion after  Corollary 2 in \cite{Daouia2011}.  The tail index being the same over all margins, we restrict our attention to $j=1$ and consider the so-called functional Hill estimator, first introduced by \cite{GG2012},
\begin{align}\label{eq:Hill2}\hat{\gamma}_{n,y}&:= \frac{\sum_{i=1}^{J}\ln {\hat{q}_{1,n,y}\left(1-\tau_i \overline{\al}_n\right)}-\ln {\hat{q}_{1,n,y}\left(\al_n\right)}}{-\sum_{i=1}^J \ln(\tau_i)}.
\end{align}
Concerning the tail ratio $c_{j,y}$ in Condition \ref{hyp:model2}, we simply extend in a functional framework the estimator proposed in \cite{BeckMailhotElena2021,MaumeDeschampsRulliereSaidExtremes}
\begin{align}\label{eq:hatc_j}
\hat{c}_{j,n,y}&:=\left( \frac{\hat{q}_{ j,n,y}\left(\al_n\right)}{\hat{q}_{ 1,n,y}\left(\al_n\right)}\right)^{1/\hat{\gamma}_{n,y}} \,\, \mbox{ for } \, \, j=2,\ldots,d.
\end{align}

By using the inference procedure presented in Sections \ref{subsection1Estimators} and \ref{subsection2Estimators}, we now focus on the convergence of the approximated optimum problem  in Equation \eqref{eq:optimization_approx}.

\section{Convergence of the approximated optimum problem for functional MEEs}\label{convergence}

In Section \ref{AssumptionsMainTheorem} below, we introduce conditions required in our main convergence result, \emph{i.e.}, Theorem \ref{prop:main}. These conditions fall into porous categories. Because of the bias, the scaling sequence $\alpha_n$ and the smoothing parameter $h_n$ are linked through different regimes involving in particular the small ball probability $\psi_y(h_n)$.  

\subsection{Required assumptions} \label{AssumptionsMainTheorem}
  
We organize the presentation of  our assumptions by separating them  in  three classes (I, II and III). The first class below  is related to the bandwidth $h_n$ for the small ball probability in \eqref{smallBall}, the risk level $\alpha_n$  and the marginal distribution function of $X_1$.

\begin{enumerate}[label=(I.\arabic*)]\label{classeI}
	\item \label{hyp:regime}  The bandwidth $h:=h_n\ll 1$ satisfies the intermediate regime $n\overline{\al}_n\psi_y(h_n) \gg 1$.
\item \label{hyp:oscillation} For any $1\le j\le d$, there exists $\delta_j>0$ such that for any $y\in E$,  
\begin{align*}
\sqrt{ n\overline{\al}_n\psi_y(h_n)} \ln(\overline{\al}_n)  \sup_{{z\ge (1-\delta_j)q_{j,y}(\al_n)}\atop{y'\in B(y,h_n)}}
\ff{\ln z} \zp[b]{\ln \f{\overline{F}_{1,y'}(z)}{\overline{F}_{1,y}(z)}}&\xrightarrow[n\to +\infty]{} 0.
\end{align*}	
 \item \label{hyp:regime2}  For any $y\in E$, it holds that  $	\sqrt{	n\overline{\al}_n\psi_y(h_n) } A_{1,y}(\overline{\al}^{-1}_n)\xrightarrow[n\to +\infty]{}p_1\in \R$, where $A_{1,y}$ is the auxiliary function of $X_1$ from Condition  \ref{hyp:model1} (see also Equation  \eqref{eq:soc1}).
	\item \label{hyp:tau} For any $s\in [0,1]$ and $0<h_n\ll 1$ deterministic, the following pointwise convergence holds
\begin{align*}
\tau_{y,h_n}(s):=\f{\psi_y(h_ns)}{\psi_y(h_n)} = \p\left(  \zp[n]{Y-y}_E <sh_n \big |  \zp[n]{Y-y}_E<h_n \right)&\xrightarrow[h_n \to 0_+]{}\tau_{y,0}(s).
		\end{align*}
\end{enumerate}\smallskip 

A second class of hypothesis involves the {conditional}  bivariate distribution  functions and the {conditional}  bivariate upper tail dependence functions.

\begin{enumerate}[label=(II.\arabic*)]\label{classeII}
\item \label{hyp:integrability} For any $1\le j<k\le d$ and $y\in E$, there exists $\eta_{j,k,y}>0 $ such that, for $n$ large enough, as $t\to 0$,\begin{align*}
			\lam_{ j,k,y}\left(t,1\right) = O\left( t^{\,\eta_{j,k,y}} \right),
					\quad \hat \lam_{j,k,n,y}\left(t,1\right) = O_{\p}\left( t^{\,\eta_{j,k,y}} \right).\end{align*}
				Furthermore, it holds that $\gamma_y< \eta_y:= \min_{j\neq k} \eta_{j,k,y}$.
\item \label{hyp:continuity} For any $1\le j<k\le d$ and $y\in E$,  it holds that $\sup\limits_{y'\in E, \,\zp[n]{y'}\le 1} \sup\limits_{\bm x\in \R^2} \zp[b]{F_{j,k,y+\delta y'}(\bm x)-F_{j,k,y}(\bm x)}\xrightarrow[\delta \to 0_+]{}0.$ 
	\item \label{hyp:beta} For any $1\le j <k\le d$, analogously as in \cite{Ferraty2007} we consider  
\begin{align*} 
	\phi^F_{j,k,\bm x,y}(s)&:=\E_Y\left[F_{j,k,Y}(\bm x)-F_{j,k,y}(\bm x) \,\big|\, \zp[n]{Y-y}_E = s \right], \,\, \mbox{ for } \,\, \bm x = (x_j,x_k),
\end{align*} quantifying  the expected difference $F_{j,k,Y}(\bm x)-F_{j,k,y}(\bm x)$
when $Y$ is forced to be at a distance $s$ from the point $y \in E$.
Assume that there exists $\beta>0$ such that the function $s \mapsto \phi^F_{j,k,\bm x,y}(s)$ with $\bm x = (x_j,x_k)$ satisfies, uniformly in $\bm x$, \begin{align*}
			\phi^F_{j,k,\bm x,y}(s) = s^{\,\beta}\nu_{j,k}(\bm x)+o(s^{\,\beta}),\quad s\to 0_+,
		\end{align*} for some continuous function $\bm u \mapsto\nu_{j,k}( F^{-1}_{j,y}(u_j),F^{-1}_{k,y}(u_k) ) $ on $[0,1]^2$. Furthermore, the bandwidth $h:=h_n\ll 1$ satisfies $ n\psi_y(h_n)h_n^{\,2\beta} = O(1)$.
	\end{enumerate}\smallskip 

 A third class hypothesis focuses on the Kernel $K$ in \eqref{eq:weight}.
\begin{enumerate}[label=(III.\arabic*)]\label{classeIII}
\item \label{hyp:kernel} $K$ has support $[0,1]$, is bounded in $(0,1)$, has a continuous derivative on $[0,1)$ with $K'(s)\le 0$ and \begin{align*}
		M_1 := K(1) - \int_0^1 K'(s)\tau_{y,0}(s)\mathrm{d}s >0,
\end{align*}	 where $\tau_{y,0}$ is the pointwise limit as $h \to 0_+$ of $\tau_{y,h}$ in \eqref{hyp:tau}. 
		\end{enumerate}

A discussion on the previous required assumptions is postponed to Section \ref{Hypdiscussion}. We are now ready to state and prove our main result.
 
\subsection{Main result}

We begin by enunciating the multiple hypotheses needed for our main result. It is then followed by the different speed rates associated to the estimation part of the optimum problem and to the actual minimization algorithm (BFGS family). These scaling sequences are expressed in terms of copula estimation, small ball probability and regular variation.

Let a covariate value $y\in E$. Regarding {the  inference}, we require that:
\begin{itemize}
    \item the random vector $\bm X=(X_1,\ldots,X_d)$ given $Y=y$ is multivariate regularly varying with index $1/\gamma_y$ {(see Appendix 
 \ref{sec:RVF})} and the conditional marginal tails satisfy  Conditions \ref{hyp:model1} and \ref{hyp:model2}. 
\item The  condition in \eqref{hyp:segers} and the first order copula condition in \eqref{hyp:rate_theoretical}. 
\item Assumptions of classes I, II and III in Section \ref{AssumptionsMainTheorem}.
\end{itemize}

Next, denote by $\hat{\bm{\Theta}}^{(m)}_y$  the $m^{\rm th}$ iteration step of the minimization algorithm applied to the approximated optimum problem \eqref{eq:optimization_approx}. We make the following assumptions on the optimization procedure: 
\begin{itemize}
    \item for any $y\in E$, there exists a hyperrectangle $\mathcal{K}:= \prod_{r=1}^d [\eps_r,M_r]\subset (0,+\infty)^d$  such that for any $m\ge 1$ and $n\ge 1$ large enough, $\hat{\bm{\Theta}}^{(m)}_y\in \mathcal{K}$.
    \item The minimization algorithm solves for the global minimum of optimum problem \eqref{eq:optimization} and has complexity  $O(\delta_{(m)})$, i.e., the computational cost of $m$ steps of the algorithm is proportional to $\delta_{(m)}$.
\end{itemize}

Concerning the speed rates, for any $1\le j\le d$, any deterministic $\al_n = 1+o(1)$, $1\ll k_n\ll n$ and $m_n\gg 1$, we define the following deterministic sequences, with $\mu_y>0$ as in \eqref{hyp:rate_theoretical},
\begin{align}\label{eq:scaling1}
	 \tilde{\delta}_{-1,n,y}&:= \left(  \f{n}{k_n}\right)^{\,\mu_y},\quad\tilde \delta_{0,n,y}:=\sqrt{n\overline{\al}_n\psi_y(h)} ,\quad \tilde \delta_{j,n,y} := A^{-1}_{j,y}\left( \overline{\al}_n^{-1}\right) .
\end{align}
Next, we introduce the speed rate $\delta_{n,y}$ as a scaling sequence such that
\begin{align}\label{eq:scaling2.5}
	 1\ll \delta_{n,y}\ll  \min_{-1\le j \le d}\tilde \delta_{j,n,y}.
\end{align}



We may now state our main result under the previous assumptions.
\begin{Th}[Approximated optimum convergence  with rate]\label{th:main2}  
\textit{Let $y \in E$.  Let $\bm\Theta^\star_y$ the optimum of \eqref{eq:optimization} and $\hat{\bm{\Theta}}^{(m_n)}_y$ the $m_n^{\rm th}$ iteration step in the minimization algorithm applied to \eqref{eq:optimization_approx} where $m_n\gg 1$ is deterministic. Let also two scaling sequences $\delta_{n,y}$ and $\delta_{(m_n)}$. Then, under the previous hypotheses,\begin{align*} 
			\min \left(\delta_{n,y}, \, \delta_{(m_n)}	\right)\zp[n]{\hat{\bm{\Theta}}^{(m_n)}_y- {\bm \Theta}^{\star}_{y}}_1&\xrightarrow[n\to +\infty]{\p}0. 
\end{align*}}
\end{Th}
The proof of Theorem \ref{th:main2} is postponed  below.

In order to estimate functional MEEs, we draw our attention on \eqref{eq:system_CEME}. A direct consequence is that, if unique, the limit $\bm \Theta=(\eta,\beta_2,\ldots,\beta_d)$ of $(  \frac{1-\al}{\overline{F}_{1,y}\left(\bm{e}_\al^1(\bm X,y)\right)},\frac{\bm{e}_\al^2(\bm X,y)}{\bm{e}_\al^1(\bm X,y)} ,\ldots, \frac{\bm{e}_\al^d(\bm X,y)}{\bm{e}_\al^1(\bm X,y)})$, as $\al \to 1$, satisfies the optimum problem \eqref{eq:optimization} and \begin{align*}
	&\bm{e}_\al( \bm X,y)\sim F^{-1}_{1,y}(\al)\eta^{\gamma_y}\left( 1,\beta_2,\ldots,\beta_d\right)^t,\quad \al \to 1.
\end{align*}

{Notice that the optimum convergence  in \cite{BeckMailhotElena2021} is given \emph{via} a non-interchangeable two limits result. The first concerns the convergence of the statistical approximation of the loss function, the second the  convergence of the descent gradient algorithm. The interested reader is referred to Algorithm 2 and Corollary 3 in \cite{BeckMailhotElena2021}.}
{Conversely,  our result unifies both limits $n,m\to +\infty$ into a single limit, i.e.,  
\begin{align}\label{eq:doublelimit2}\hat{\bm \Theta}^{(m_n)}_{y}\xrightarrow[n\to +\infty]{\p}  {\bm \Theta}^{\star}_{y}.\end{align} 

The speed rate $\delta_{n,y}$ in Theorem \ref{th:main2}, is associated to the  statistical approximation of the loss function. The second rate $\delta_{(m)}$ is inherited of the used minimization algorithm. Considering the double limit $m:=m_n\gg 1$, the global speed rate of (\ref{eq:doublelimit2}) naturally involves a compromise between the two scaling sequences. For sake of clarity, in Remark \ref{remarkscales} below, we present some  more practical expressions on the  speed rate  $\delta_{n,y}$. 
 
\begin{rmk}[About the rate $\delta_{n,y}$]\label{remarkscales}
Under Assumption \ref{hyp:regime2}, the condition on the scaling sequence $\delta_{n,y}$ in \eqref{eq:scaling2.5}  becomes
	\begin{align}\label{eq:scaling4}
		1\ll \delta_{n,y}\ll  \min\limits_{\substack{-1\le j \le d\\j\ne -1}}\tilde \delta_{j,n,y}.
	\end{align} 
Notice that explicit expressions for the auxiliary functions $A_{k,y}$ are available (e.g., \cite{Mao}) for particular distributions. 
For instance,  
one can consider  the Lomax (or Pareto II) conditional marginal distributions, which we generically denote by $X_j|Y=y\sim {\rm Lomax}(1/\gamma_y,s_j)$ where $\gamma_y,s_j>0$ for $y\in E$ and $1\le j \le d$. For this model, the survival {conditional distribution} function is $\overline{F}_{j,y}(x) = (1+x/s_j)^{-1/\gamma_y} \in {\rm 2RV}_{-1/\gamma_y,-1},$ 
or equivalently, the {conditional} tail quantile function is $U_{j,y}(t)=s_j(t^{{\gamma_y}}-1)\in {\rm 2RV}_{\gamma_y, -\gamma_y}$   with auxiliary function $A_{y}(t)= {\gamma_y}t^{-{\gamma_y}}$. In this case, the speed rate in \eqref{eq:scaling4} has a simpler expression \begin{align*}
		1\ll \delta_{n,y}\ll  \min\left\{ \delta_{-1,n,y},\delta_{0,n,y},{\gamma_y}(1-\al_n)^{-{\gamma_y}} \right\}.
\end{align*}

Further, if one considers $k_n = n\psi_y(h_n)$ which satisfies $1\ll k_n\ll n$, one may write $ \min(\delta_{0,n,y},\, \delta_{-1,n,y}) =(n\psi_y(h_n))^{ \min(\ff{2},\, \mu_y) }$ and the relation $$ 1\ll \delta_{n,y}\ll \min \left({\gamma_y}(1-\al_n)^{-{\gamma_y}} ,(n\psi_y(h_n))^{ \min( \ff{2},\, \mu_y) }\right). $$

Another possibility is to suppress the auxiliary functions dependence by considering  $ \delta_{0,n,y} = \delta_{-1,n,y} (1+o(1))$ as $n\to +\infty$, or equivalently $k_n \sim n(n\psi_y(h_n))^{-1/(2\mu_y)}$. This particular choice of $k_n$ gives under the condition $1\ll k_n \ll n$ the new regime   $ \psi_y(h_n)^{1/{(2\mu_y)}} \ll n^{1-1/{(2\mu_y)}} $.  
\end{rmk}

\begin{rmk}[About the minimization algorithm and rate $\delta_{(m_n)}$]\label{rmk:algo} At the moment, we are essentially facing two possibilities for quasi-Newton methods (which have less complexity than the classic Newton ones). The first one is the classic BFGS method as described by Algorithm 6.1 in \cite{nocedal2006} or Algorithm 1 in \cite{BeckMailhotElena2021}. For such algorithm, each iteration costs $O(d^2)$ arithmetic operations. Hence, the whole computational cost is $O(d^2m)$ where $m$ is the number of iterations, and we set $\delta_{(m)}= d^2m$. On the other hand, one may preferably use the L-BFGS-B version,  which is designed for bound constrained optimization with limited memory storage. The motivations for this particular algorithm are threefold. First, it seems reasonable to consider MEEs only in a certain range. Second, no informations about the Hessian matrix or about the structure of the loss function are required. Besides, when $m:=m_n\gg 1$ and $n$ is the number of observations, the upcoming proof of Theorem \ref{th:main2} requires uniform convergence results. The computational cost of one iteration in the L-BFGS-B algorithm is linear in $d$, namely $O(p^2d)$ with $3\le p \le 20$ controlling the amount of storage (see \cite{Zhu}). Thus, we associate the speed rate $\delta_{(m)} = p^2dm$ to this case.
\end{rmk}

To prove  Theorem \ref{prop:main}, we proceed in a two-step demonstration.


\begin{itemize}\begin{sloppypar}
  \item[i)] Firstly, we build the approximated minimization problem  in \eqref{eq:optimization_approx} by considering an estimator vector  $\bm{\hat \xi}_{n,y} = (\hat \gamma_{n,y},\hat c_{2,n,y},\ldots,\hat c_{d,n,y},\hat \lam_y)$. Once the consistency of $\bm{\hat \xi}_{n,y}$ is assured, we prove uniform convergence with associated speed rate of the approximated loss function in Definition \ref{defLossApprox} and its gradient function to their theoretical counterparts.  \end{sloppypar}
  \item[ii)]  Then we  link the minimization algorithm with the estimation procedure by inducing a dependency of the iteration steps on the data sample size $n$. Namely, we allow $m$ to depend on $n$ so that $m=m_n\gg 1$. This translates the fact that we are simultaneously consistently approximating the optimum problem \eqref{eq:optimization} and solving the resulting approximated optimum problem \eqref{eq:optimization_approx} using the BFGS minimization algorithm. 
\end{itemize}

\begin{proof}[Proof of Theorem \ref{th:main2}]
The proof of Theorem \ref{th:main2} heavily relies on the convergence of the loss function and its gradient Moreover, the validity of step ii) can be proved  if the gradient convergence is uniform. These results are gathered in the following intermediate proposition.  The proof of Proposition \ref{prop:main} is postponed to Section \ref{sectionProof}.  
 
\begin{Prop}[Approximated loss {function} convergence  with rate]\label{prop:main}\textit{Assume that assumptions of Theorem \ref{prop:main} hold true. Then,
	\begin{align}\label{eq:main}
		\delta_{n,y} 	\zp[b]{\mathcal{L}_{\bm \xi_y} ( \bm \Theta  )- \mathcal{L}_{\hat{ \bm \xi}_{n,y}} ( \bm \Theta  )}&\xrightarrow[n\to +\infty]{\p}0,
	\end{align}
	\begin{align}\label{eq:main_gradient}
		\delta_{n,y} 	\zp[n]{\nabla\mathcal{L}_{\bm \xi_y} ( \bm \Theta  )- \nabla \mathcal{L}_{\hat{ \bm \xi}_{n,y}} ( \bm \Theta  )}_1&\xrightarrow[n\to +\infty]{\p}0.
\end{align}
Furthermore if we restrict $\bm \Theta \in \prod_{j=1}^d [\eps_j,M_j]$ with $0<\eps_j<M_j<+\infty$ for any $1\le j\le d$, we get \begin{align}\label{eq:mainsup11}
	\delta_{n,y} \sup_{\bm \Theta \in \prod_{j=1}^d [\eps_j,M_j]}	\zp[b]{\mathcal{L}_{\bm \xi_y} ( \bm \Theta  )- \mathcal{L}_{\hat{ \bm \xi}_{n,y}} ( \bm \Theta  )}&\xrightarrow[n\to +\infty]{\p}0,
\end{align}
	\begin{align}\label{eq:mainsup12}
		\delta_{n,y} \sup_{\bm \Theta \in \prod_{j=1}^d [\eps_j,M_j]}	\zp[n]{\nabla\mathcal{L}_{\bm \xi_y} ( \bm \Theta  )- \nabla \mathcal{L}_{\hat{ \bm \xi}_{n,y}} ( \bm \Theta  )}_1&\xrightarrow[n\to +\infty]{\p}0.
\end{align} }
\end{Prop}  
{Morever, in Theorem \ref{th:main2},} we assume that the minimization algorithm solves for the global minimum of \eqref{eq:optimization}. Once \eqref{eq:mainsup11} and \eqref{eq:mainsup12}  are established {(see also Equation \eqref{eq:main3sup})}, one may decompose with the triangular inequality
\begin{align*}
	\zp[n]{\hat{\bm \Theta}^{(m_n)}_{y}- {\bm \Theta}^{\star}_{y}}_1&\le 	\zp[n]{\hat{\bm \Theta}^{(m_n)}_{y}- {\bm \Theta}^{(m_n)}_{y}}_1 + 	\zp[n]{{\bm \Theta}^{(m_n)}_{y}- {\bm \Theta}^{\star}_{y}}_1
\end{align*}

{The first quantity translates the cost of estimating the $m_n^{\text{th}}$ step of the minimization algorithm,  while the second quantity is purely inherent to the nature of the used  minimization algorithm.} Concerning the minimization algorithm, we may consider any suitable gradient descent algorithm with complexity  $\delta_{(m)}$, so the whole cost is 
\begin{align*}
	\zp[n]{{\bm \Theta}^{(m_n)}_{y}- {\bm \Theta}^{\star}_{y}}_1&= O(\delta^{-1}_{(m_n)}).
\end{align*}
 
By using the classical steps in the BFGS quasi-Newton optimization algorithm (see, e.g., Algorithm 1 in \cite{BeckMailhotElena2021}), once the uniform convergence \eqref{eq:mainsup12} holds, it immediately follows that, independently of $m_n$, $$	\zp[n]{\hat{\bm \Theta}^{(m_n)}_{y}- {\bm \Theta}^{(m_n)}_{y}}_1 	= O_{\p}\left( \delta^{-1}_{n,y}\right). $$
Hence the desired result. 
\end{proof}

\subsection{Hypothesis discussion}\label{Hypdiscussion}

Quantiles are the building blocks of the estimation procedure as they play a crucial role in  the estimation of the tail index and ratios. If Conditions \ref{hyp:model1}, \ref{hyp:kernel}, \ref{hyp:regime}, \ref{hyp:oscillation} are satisfied, one can prove the following asymptotic normality of the extreme conditional quantile estimator (see for instance Proposition 1 in \cite{GSU2022} and Theorem 2 in \cite{GG2012}),
\begin{align}\label{eq:quantile}\tilde \delta_{0,n,y} \left( \frac{\hat {q}_{k,n,y}({\al}_n)}{ {q}_{k,y}({\al}_n)} -1\right)&\xrightarrow[n\to +\infty]{d} \mathcal{N}(0,1),\quad 1\le k \le d,
\end{align} 
where $\tilde \delta_{0,n,y}$ is as in \eqref{eq:scaling1}.
If in addition \ref{hyp:regime2} is fulfilled, \cite[Theorem 7]{GSU2022} extends the Gaussian fluctuations for the functional tail index (see also Theorem 4 in \cite{GG2012}),
\begin{align}\label{eq:gamma}\tilde \delta_{0,n,y}\left( \hat{\gamma}_{n,y}-\gamma_y\right) &\xrightarrow[n\to +\infty]{d} \mathcal{N}\Big( \f{p_1}{\ln (J!)}\sum_{j=2}^J \f{j^{\rho_{1,y}}-1}{\rho_{1,y}}, \f{J(J-1)(2J-1)}{6\log(J!)}\gamma^2(y) \Big),
\end{align}
with $\tilde \delta_{0,n,y}$ as in \eqref{eq:scaling1}, $\gamma_y$ as in  Condition \ref{hyp:model1} with associated  estimator $\hat \gamma_{n,y}$ as in \eqref{eq:Hill2}, $\rho_{1,y}$  the second order tail index from Condition \ref{hyp:model1}, $p_1\in \R$ as in  Condition \ref{hyp:regime2} and $J$ a positive integer (e.g., $J=9$).

Concerning the fluctuations of the empirical copula in this functional setting, one may invoke \cite[Theorem 3]{irene}. When each marginal $F_{j,y}$ is continuous and Assumptions \eqref{hyp:segers}, \ref{hyp:regime}, \ref{hyp:tau}, \ref{hyp:continuity}, \ref{hyp:beta}, \ref{hyp:kernel} are satisfied, it yields
\begin{align}\label{eq:irene}\sqrt{n\psi_y(h_n)} \sup_{\bm u \in [0,1]^2} \zp[b]{ \hat C_{j,k,n,y}(\bm u)   - C_{j,k,y} (\bm u)  } &=O_{\p}(1),
\end{align} where $\psi_y(h_n)$ is the small probability in  \eqref{smallBall}, $C_{j,k,y}$ the conditional bivariate copula (of the margins $j,k$) with associated  estimator $\hat C_{j,k,n,y}$ defined in \eqref{eq:copula_estimator}.
\\
We devote the rest of this section to illustrate the role of the required hypothesis 
in our Theorem \ref{th:main2}, in particular in order to guarantee the  crucial convergences in \eqref{eq:quantile}-\eqref{eq:gamma}-\eqref{eq:irene}.

\begin{itemize}
\item  In view of \cite[Lemma 4]{GG2012}, the regime 
 in \ref{hyp:regime} is a necessary and sufficient condition for the almost sure presence of at least one sample point in the region $B(y,h_n)\times (q_{j,y}({\al}_n),+\infty)$ for any $1\le j \le d$, that is, $q_{j,y}(\al_n)$ is within the sample. Therefore, it makes sure that, as the number of data sample points increases, the number of observations larger than the conditional quantile increases.
\item  The hypothesis \ref{hyp:kernel} gathers conditions on the kernel for the simultaneous application of convergences in  \eqref{eq:quantile}-\eqref{eq:gamma}-\eqref{eq:irene}. It is hence a combination of Type I kernel, according to \cite[Definition 4.1]{FerratyVieu2006}, required for \eqref{eq:quantile}-\eqref{eq:gamma}, and of continuous kernel versions.  
\item  Concerning   univariate quantiles and tail index, we require Conditions \ref{hyp:oscillation} and \ref{hyp:regime2} for applying \eqref{eq:quantile}-\eqref{eq:gamma} and proving Proposition \ref{lem:speed_integral} and Lemma \ref{lem:speed_beta}, \ref{lem:s_n-t}, \ref{lem:lam_conv}. The first Condition \ref{hyp:oscillation} allows to control the
oscillations of the survival  distribution function when the covariate is varying.  For a deeper discussion, we refer to the end of Section 2 in \cite{GSU2022}. On the other hand, as well-known in the extreme literature, the  Hill type estimators require a second order regular variation condition (see also Condition \ref{hyp:model2}) in order to determine their rate convergence and fluctuations. The estimator in \eqref{eq:Hill2} of the tail index is built upon the first marginal empirical quantile without any second order related object involved. This lacking in the tail approximation generates a bias that we should assume \emph{via} \ref{hyp:regime2} to be negligible in the scale $\tilde \delta_{n,y,0}$. Inherited from Theorem 7 in \cite{GSU2022}, Condition \ref{hyp:regime2} links the asymptotic of the first marginal auxiliary function $A_{1,y}$ with the small ball probability regime.
\item  Condition \ref{hyp:integrability} is an integrability requirement, used in Proposition \ref{lem:speed_integral}, for manipulating conditional stable tail dependence functions and their estimators as they appear through integrals in the loss function \eqref{eq:loss}. 
\item Assumptions \ref{hyp:tau}, \ref{hyp:continuity}, \ref{hyp:beta} are technical requirements for \eqref{eq:irene} related to the functional aspect and appearing in Proposition \ref{lem:speed_integral} and Lemma \ref{lem:lam_conv}.  For instance, the authors of \cite{irene} used  Condition \ref{hyp:continuity} for tightness and equicontinuity in the considered empirical processes.  In particular, the object $\tau_{y,h}$  in Assumption \ref{hyp:tau} was first introduced in \cite{Ferraty2007} with the same condition on the limiting function $\tau_{y,0}$. {The connection with assumption on $M_1$ in  Condition \ref{hyp:kernel} comes down from the fact that kernel expectations are key objects (see \cite{GG2012,GSU2022}) in functional kernel methods and may be written as $\E_Y[ K(h^{-1}_n{\zp[n]{Y-y}_E}) ]= K(1)\psi_y(h_n) - \int_0^1K'(s)\psi_y(s h_n)\mathrm{d}s$ (see \cite[Lemma 2]{Ferraty2007}).} The identification of $\tau_{y,0}$ is driven by the asymptotic behaviour of $s\mapsto \psi_y(s):=\p\left(\zp[n]{Y-y}_E\le s  \right)$.  As stated in \cite{FerratyVieu2006}, the class of fractal (infinite-dimensional) processes is the one for which the small ball probabilities behave as
$\psi_y(s) \sim C_ys^a$, for ${s\to 0}$ and $C_y,a>0$. For such processes, we have $\tau_{y,0}(s)=s^a$. This includes the multivariate Euclidean  case $E=\R^p$, with $a=p$. On the other hand, for the  class of   processes such that $\psi_y(s) \sim C_ys^a\exp\big( -C/s^b \big) $, for ${s\to 0}$ and  some $C_y,C,a,b>0$, one has $\tau_{y,0}(s) = \delta_1(s)$ with $\delta_1$ being the Dirac delta function at point $1$ (see   \cite{Ferraty2007}). As noted in \cite{irene}, for those processes, $\int_0^1 K'(s)\tau_{y,0}(s)\,\mathrm{d}s$ is zero so that one usually assumes $K(1)>0$ in this case. In the general case, the condition $K'(\cdot)\le 0$ is asked. For more details about $\tau_{y,h}$ and $\beta$ in  Condition \ref{hyp:beta}, the interested reader is referred  to Section 2.3, especially Proposition 1 and Section 4, in \cite{Ferraty2007}, and Section 3.2 in \cite{irene}.
\end{itemize}

\section{Proof of Proposition \ref{prop:main}}\label{sectionProof}

Soliciting $\Delta$-method (see, e.g.,  Theorem 3.8 in \cite{vaart_1998}), the proof of \eqref{eq:main} and \eqref{eq:main_gradient} is reduced to the establishment of the building block  	\begin{align}\label{eq:main3}
	\delta_{n,y}  \max_{1\le k\le d}	\zp[b]{	{\vfi}_{k,y} ( \bm \Theta ,\bm \xi_y ) -{\vfi}_{k,n,y} ( \bm \Theta ,\bm{\hat \xi}_{n,y} )}&\xrightarrow[n\to +\infty]{\p}0.
\end{align}
The proof of  \eqref{eq:mainsup11} and \eqref{eq:mainsup12} relies upon, for $0<\eps_r<M_r<+\infty$, $1\le r\le d$,
\begin{align}\label{eq:main3sup}
	\delta_{n,y}  \max_{1\le k\le d}	\sup_{\bm \Theta\in \prod_{{r}=1}^d[\eps_{{r}},M_{{r}}] }\zp[b]{	{\vfi}_{k,y} ( \bm \Theta ,\bm \xi_y ) -{\vfi}_{k,n,y} ( \bm \Theta ,\bm{\hat \xi}_{n,y} )}&\xrightarrow[n\to +\infty]{\p}0.
\end{align}


 In view of \eqref{eq:vfi}, the proof of \eqref{eq:main3} and \eqref{eq:main3sup} comes down by proving the following  convergence
\begin{align}	\label{eq:boil}
	\delta_{n,y} \max \left( \zp[b]{\Delta^{(1)}_{n,y}} ,\zp[b]{\Delta^{(2)}_{k,n,y}(\bm \Theta )},\zp[b]{\Delta^{(3)}_{j,k,n,y}(\bm \Theta )}\right)\xrightarrow[n\to +\infty]{\p} 0,\quad 1\le j\ne k\le d, \quad y\in E,
\end{align}
and its uniform counterpart 
\begin{align}	\label{eq:boilsup}		\delta_{n,y} \sup_{\bm \Theta \in \prod_{{{r}}=1}^{d}[\eps_{{r}},M_{{r}}]} \max \left( \zp[b]{\Delta^{(1)}_{n,y}} , \zp[b]{\Delta^{(2)}_{k,n,y}(\bm \Theta )},\zp[b]{\Delta^{(3)}_{j,k,n,y}(\bm \Theta )}\right)\xrightarrow[n\to +\infty]{\p} 0,
	\end{align}

where  $\bm \Theta = (\eta,\beta_2,\ldots,\beta_d)\in (0,+\infty)^d$, $\beta_1=1$ and 
\begin{align*}
\Delta^{(1)}_{n,y} &:={\frac{\gamma_y}{\gamma_y-1}-\frac{\hat \gamma_{n,y}}{\hat \gamma_{n,y}-1}},\quad
\Delta^{(2)}_{k,n,y}( \bm \Theta) = \eta \left(1+\sum_{j\neq k} \frac{\beta_j}{\beta_k} \right)\left( {\frac{\beta^{1/{\gamma_y}}_k}{c_{k,y}}-\frac{\beta^{1/\hat \gamma_{n,y}}_k}{\hat c_{k,n,y}}}\right) ,
\end{align*}
\begin{align}\label{eq:delta3}
\Delta^{(3)}_{j,k,n,y}(\bm \Theta )= \sum_{j\neq k} \left(\int_{\frac{\beta_j}{\beta_k}}^{+\infty}\hat \lam_{ j,k,n,y}\left(\frac{\hat c_{j,n,y}}{\hat c_{k,n,y}}t^{-1/{\hat \gamma_{n,y}}},1\right) - \lambda_{ j,k,y}\left(\frac{c_{j,y}}{c_{k,y}}t^{-1/{\gamma_y}},1\right) \mathrm{d}t\right).
\end{align} 

A straightforward application of the $\Delta$-method and  \eqref{eq:gamma} implies that for any deterministic $1\ll \delta_{0,n,y}\ll \tilde\delta_{0,n,y}$ and independently on $\bm \Theta\in (0,+\infty)^d$, $$\delta_{0,n,y} \zp[b]{\Delta^{(1)}_{n,y}} \xrightarrow[n\to +\infty]{\p} 0.$$

 This readily provides \eqref{eq:boil} and \eqref{eq:boilsup} for $\Delta^{(1)}_{n,y}$. The quantities $\Delta^{(2)}_{k,n,y}$ and $\Delta^{(3)}_{j,k,n,y}$ are respectively tackled in the upcoming Lemma \ref{lem:speed_beta} and Proposition \ref{lem:speed_integral}. One can obtain the identities for the gradient loss function by adapting the setting in   Appendix A.3.  in  \cite{BeckMailhotElena2021} to our functional covariate case. By using these gradient expressions we can  show the  gradient convergence in \eqref{eq:mainsup11}-\eqref{eq:mainsup12}.  Indeed, the random quantities $\hat \gamma_y$ and $\hat c_{k,y}$ appear through algebraic expressions in the  gradient loss function in the same form as in \eqref{eq:vfi}.  \\
 
The Karamata's results, Theorem B.1.4 - B.1.5, Equation (B.1.14), Remark B.1.7 in \cite{extreme_value_theory_de_haan_ferreira}, see also Theorem 1.3.1 in \cite{bingham_goldie_teugels_1987}, allow us to write
\begin{align}\label{eq:ell_karamata}\ell_{U_{j,y}}(x)&:=\ell_{j,y}(x)=\kappa_j(y)  b_{j,y}(x) \exp \left(\int_{x_0}^x \frac{ b_{j,y}(t)-1}{t}\mathrm{d}t \right),
\end{align}
for some arbitrary $x_0>0$, some constant $\kappa_j(y)\in (0,+\infty)$ and \begin{align}\label{eq:b} b_{j,y}(x) =\f{ x\ell_{j,y}(x)}{\int_{x_0}^x \ell_{j,y}(t)\mathrm{d}t}.
\end{align}
	
Studying the Karamata representation \eqref{eq:ell_karamata} of the slowly varying part $ \ell_{j,y}$ through the point of view of $\Pi$-Class (see Definition \ref{def:pi}), we may show in the following Lemma \ref{lem:slowly_asymptotic_rho}  an asymptotic expansion on $\ell_{j,y}$.\\ 

 Notice that Lemma \ref{lem:slowly_asymptotic_rho} in the case $\rho_{j,y} < 0$ is a straightforward generalization to the conditional setting of Lemma 1 in \cite{HuaJoe2011}. The proof in the specific case where $\rho_{j,y} = 0$ is  postponed to Section \ref{sectionProofAUX}. For examples of 
 distributions satisfying the second order regular variation condition with  $\rho_{j,y} = 0$  we interested reader is referred to Remark \ref{exampleFbar}.

\begin{lem}\label{lem:slowly_asymptotic_rho} Assume that for any $y\in E$, for any $1\le j\le d$, $ U_{j,y}\in 2{\rm{RV}}_{-{\gamma_y},\rho_{j,y}}(+\infty),\, \gamma_y>0, \, \rho_{j,y} \le 0, $ with auxiliary function $A_{j,y}\in {\rm{RV}}_{\rho_{j,y}}(+\infty)$ or equivalently Condition \ref{hyp:model1}.
	Then, for some constant $ \kappa_j(y)\in (0,+\infty)$ independent of~$x$, \begin{align*}\ell_{j,y}(x)&=\kappa_j(y)  \left(1+ O\left(  A_{j,y}(x)\right)  \right), \quad x\to +\infty.
	\end{align*}
\end{lem} 
Once the slowly varying part of the tail quantile function is determined, we may derive, under the second order regular variation property (see Assumption \ref{hyp:model2}), the speed rate in the consistency of the tail ratios. 

\begin{lem}\label{lem:speed_beta}
Let  $\bm \Theta= (\eta,\beta_2,\ldots,\beta_d)>0$, $y \in E$,   $\al_n=1+o(1)$, $ h_n \ll 1 $ deterministic such that the intermediate regime $n\overline{\al}_n\psi_y(h_n)\gg 1$ holds true.    For each marginal $2\le k \le d$, let $\delta_{k,n,y}$  be a scaling sequence  such that
\begin{align}\label{eq:scaling2}
	&1\ll \delta_{k,n,y} \ll   \min_{s\in \{0,1,k\}}  \tilde \delta_{s,n,y},\hspace{0.2cm} \text{where $\tilde \delta_{s,n,y}$ is defined in Equation \eqref{eq:scaling1}}.
 \end{align}  Then, as $n\to +\infty$, under Assumptions  \ref{hyp:model1}, \ref{hyp:model2}, \ref{hyp:kernel}, \ref{hyp:regime}, \ref{hyp:oscillation}, \ref{hyp:regime2},
	\begin{align}\label{eq:rate_c}\delta_{k,n,y} \zp[b]{ c_{k,y}-\hat c_{k,y}} &\xrightarrow[n\to+\infty]{\p} 0.
	\end{align}

Furthermore,  it holds that, {for $1 \le k \le d$,}
	\begin{align*}
	\delta_{k,n,y} \zp[b]{\Delta^{(2)}_{k,n,y}(\bm \Theta)}&\xrightarrow[n\to+\infty]{\p} 0,\quad 	\delta_{k,n,y} \hspace{-0.25cm}\sup_{\bm \Theta\in \prod_{r=1}^d [\eps_r,M_r]}\zp[b]{\Delta^{(2)}_{k,n,y}(\bm \Theta)}\xrightarrow[n\to+\infty]{\p} 0,
	\end{align*}
 for $0<\eps_r<M_r<+\infty$, $1\le r \le d$. 
\end{lem}
As defined in \eqref{eq:scaling2}, the auxiliary marginal scaling sequences $\delta_{k,n,y}$ may be combined to recover the global rate $\delta_{n,y}$ defined in \eqref{eq:scaling2.5} by considering the minimum over $\tilde \delta_{-1,n,y}$ and over every $\delta_{k,n,y}$ for $k\ge 2$.  The proof of  Lemma \ref{lem:speed_beta} is postponed to Section \ref{sectionProofAUX}. Next, we consider the integral of stable tail dependence functions and prove the following convergences.
\begin{Prop}\label{lem:speed_integral}  Let  $\delta_{n,y}$ as in \eqref{eq:scaling2.5}. Let $y\in E$, $1\le j\neq k\le d$, $\bm \Theta=(\eta,\beta_2,\ldots,\beta_d)\in (0,+\infty)^d$ and $\beta_1=1$. If   Assumptions   \ref{hyp:model1}, \ref{hyp:model2}, \eqref{hyp:segers}, \eqref{hyp:rate_theoretical} and Assumptions of classes I, II, III in Section \ref{AssumptionsMainTheorem} are satisfied, then 
 $$\delta_{n,y}\zp[b]{ \Delta^{(3)}_{j,k,n,y}(\bm \Theta)}\xrightarrow[n\to +\infty]{\p}0,  $$ and for any $0<\eps_r<M_r<+\infty$, $1\le r\le d$,
	$$\delta_{n,y} \sup_{\bm \Theta \in \prod_{r=1}^{d}[\eps_r,M_r]}\zp[b]{ \Delta^{(3)}_{j,k,n,y}(\bm \Theta)}\xrightarrow[n\to +\infty]{\p}0.  $$
\end{Prop}
The proof of Proposition \ref{lem:speed_integral} can be found in Section \ref{sectionProofAUX}.

  
\section{Auxiliary proofs and results}\label{annexeA}
We gather in this section the proofs of the results needed for Proposition \ref{prop:main} as well as further auxiliary results required for the intermediate steps of the proof.  
\subsection{Proofs of lemmas of  Section \ref{sectionProof}}\label{sectionProofAUX}

We introduce a particular proxy version to the second order regular variation property to be connected to slowly varying functions when the second order tail index is zero. In this context, we constructively specify the auxiliary function and shall make use of an uniform Drees-type inequality. This is of importance in order to address in depth the Karamata representation applied to the slowly varying part of the conditional marginal quantile tail function.

\begin{definition}[Class $\Pi$][Definition B.2.4 in \cite{extreme_value_theory_de_haan_ferreira}]\label{def:pi}
	A measurable function $f:\R^+\to \R$ belongs to the class $\Pi(a)$ if there exists an auxiliary function $a:\R^+ \to \R^+$ such that for $x>0$,\begin{align}\label{eq:pi_def}\lim\limits_{t\to +\infty}\f{f(tx)-f(t)}{a(t)}&=\ln(x).
	\end{align}
\end{definition}

\begin{rmk}\label{rmk:aux}
The auxiliary function is slowly varying, namely $a\in {\rm{RV}}_0(+\infty)$, by Theorem B.2.7 in \cite{extreme_value_theory_de_haan_ferreira}. By Theorem B.2.12 in \cite{extreme_value_theory_de_haan_ferreira}, $f\in \Pi(a)$ is equivalent to $f\in \Pi(\tilde{a})$ with $\tilde{a}(t) = f(t)-\ff{t}\int_0^tf(s)\mathrm{d}s$. Hence, if $f\in \Pi(a)$ for some auxiliary function $a$, we can always choose instead the alternative auxiliary function $\tilde{a}$. More can be said, if in addition $x^{{\gamma}}f(x) \in 2{\rm{RV}}_{{\gamma},0}$, $\gamma>0$. In this case, the auxiliary function $A$ in Definition \ref{def:soc1} can be chosen as $A(t)=\tilde{A}(t):=\f{\tilde{a}(t)}{f(t)}$. According to Proposition B.2.17 in \cite{extreme_value_theory_de_haan_ferreira}, if $f\in \Pi(a)$, then for any $\eps,\delta>0$, there exists $x_0=x_0(\eps,\delta)\ge 0$, such that for any $t$ and $x$ with $tx\ge x_0$, the following uniform bound holds
\begin{align}\label{eq:bound_pi} \zp[b]{\f{\ell(tx)-\ell(x)}{\tilde{a}(x)}-\ln(t)}&\le \eps \left(\max \left( t^{\delta},t^{-\delta}\right)\right) ,
\end{align} with the relation $\tilde{A} \ell =\tilde{a}$.
\end{rmk}

\begin{proof}[Proof of Lemma \ref{lem:slowly_asymptotic_rho}]
Let $1\le j \le d$ and $y\in E$. The case $\rho_{j,y} < 0$ is a straightforward generalization to the conditional setting of Lemma 1 in \cite{HuaJoe2011}. We may now assume the second order tail index of $U_{j,y}$ to be zero, i.e.,  $\rho_{j,y}=0.$

According to the definition of the second order regular variation, there exists auxiliary functions $A_{j,y}:=A_{U_{j,y}}\in {\rm{RV}}_0(+\infty)$ for $1\le j \le d$ with ultimately constant sign and $\lim\limits_{t\to +\infty}A_{j,y}(t)=0$ such that
	$$ 	\lim\limits_{t\to +\infty}	\ff{A_{j,y}(t) \ell_{j,y}(t)} \left( \ell_{j,y}(tx)-\ell_{j,y}(t) \right)=\ln(x),\quad  x>0.$$
 
Since $\ell_{j,y}$ is non-negative, we immediately deduce that $$ \ell_{j,y} \in \Pi(a_{j,y}),\quad a_{j,y}(t):=  {\rm{sign}} \left\{  A_{j,y}(t) \right\} A_{j,y}(t)  \ell_{j,y}(t) \in {\rm{RV}}_0(+\infty),\quad a_{j,y}\ge 0.$$
	
By Remark \ref{rmk:aux}, the definition of the second order regular variation and the relation ${U_{j,y}}(x)=x^{\gamma_y}\ell_{j,y}(x)$, we can modify the auxiliary functions so that $$ \ell_{j,y}\left( \cdot \right)  \in \Pi(\tilde{a}_{j,y}),\quad \tilde{a}_{j,y}(t):= \ell_{j,y}(t)-\ff{t}\int_0^t\ell_{j,y}\left( s \right)\mathrm{d}s,$$
and ${U_{j,y}}\left( \cdot \right)\in 2 {\rm{RV}}_{-1/{\gamma_y},0}(+\infty)$, with auxiliary function $$ A_{j,y}(t) = \f{\tilde{a}_{j,y}(t)}{\ell_{j,y}(t)}= 1 - \int_0^1 \f{\ell_{j,y}(tu)}{\ell_{j,y}(t)}\mathrm{d}u\in {\rm{RV}}_{0}(+\infty).$$
	
Fix $\eps>0$ and $\delta \in (0,1)$. Let $x_0(\eps,\delta)>0$ such that \eqref{eq:bound_pi} holds. From Potter's Theorem B.1.9 in \cite{extreme_value_theory_de_haan_ferreira} with $\delta_1=1$ and $\delta_2=1/{2}$, let $x_0(\delta_1,\delta_2)>0$ such that Potter's bounds hold as well. We finally set $x_0 := \max\left( x_0(\eps,\delta), x_0(\delta_1,\delta_2)\right)$. In the rest of the proof, we consider this particular choice of $x_0$ in \eqref{eq:ell_karamata} and investigate \eqref{eq:b}.
	
With the substitution $u=\f{t}{x}$ in \eqref{eq:b}, one can express $b_{j,y}^{-1}(x) = \int_{\f{x_0}{x}}^1 \f{\ell_{j,y}(ux)}{\ell_{j,y}(x)}\mathrm{d}u$. Then, with $A_{j,y}\in {\rm{RV}}_{0}(+\infty)$, $$\ff{A_{j,y}(x)}\left(b^{-1}_{j,y}(x)-1+\f{x_0}{x} \right) = \int_{0}^1\ff{A_{j,y}(x)}\left[ \f{\ell_{j,y}(ux)}{\ell_{j,y}(x)}-1\right]\bm{1}_{ [\f{x_0}{x},1] }(u)\mathrm{d}u  . $$

For $u\in (0,1)$ and $x$ such that $ux\ge x_0$, the bound \eqref{eq:bound_pi} yields  $$\zp[b]{\ff{A_{j,y}(x)}\left[ \f{\ell_{j,y}(ux)}{\ell_{j,y}(x)}-1\right]\bm{1}_{ [\f{x_0}{x},1] }(u)}\le -\ln(u) +\eps u^{-\delta}:=f_{\eps,\delta}(u).  $$
	
The function $u\mapsto f_{\eps,\delta}(u)$ is integrable on $(0,1)$ as $-1<-\delta< 0$. Altogether with the pointwise convergence \eqref{eq:pi_def}, the dominated convergence theorem gives
$$\int_{0}^1\ff{A_{j,y}(x)}\left[ \f{\ell_{j,y}(ux)}{\ell_{j,y}(x)}-1\right]\bm{1}_{ [\f{x_0}{x},1] }(u)\mathrm{d}u  \xrightarrow[x\to +\infty]{}\kappa_{0} :=\int_0^1 \ln(u)\mathrm{d}u = -1 < +\infty  .$$
	
The previous convergence holds with limiting constant $-\kappa_{0}$ instead of $\kappa_{0}$ by accordingly adjusting the sign of the auxiliary function $A_{j,y}$. Opting for this convention, we subsequently set $\kappa_{0}=1$. One may then write $$b^{-1}_{j,y}(x)= 1+ A_{j,y}(x) -\f{x_0}{x}+ o\left( A_{j,y}(x)\right)  .   $$
	
Since $x\mapsto \ff{x}\in {\rm{RV}}_{-1}(+\infty)$ and $A_{j,y}\in {\rm{RV}}_{0}(+\infty)$, it follows that $b_{j,y}(x) = 1+ A_{j,y}(x) + o\left( A_{j,y}(x)\right) ,\quad x\to +\infty .$
	
Going back to \eqref{eq:ell_karamata},\begin{align*}\ell_{j,y}(x)&=\kappa_j(y)  \left(  1+ A_{j,y}(x) + o\left(  A_{j,y}(x)\right) \right)  \exp \left(\int_{x_0}^x \frac{   A_{j,y}(t) + o\left( A_{j,y}(x)\right)  }{t}\mathrm{d}t \right).
	\end{align*}
	
Repeating the same argument of dominated convergence but in the first order regular variation context for the auxiliary function $A_{j,y}\in {\rm{RV}}_{0}(+\infty)$ with Potter's Theorem B.1.9 in \cite{extreme_value_theory_de_haan_ferreira}, $$ \int_{x_0}^x \frac{  A_{j,y}(t) }{t}\mathrm{d}t =  A_{j,y}(x)+o\left( A_{j,y}(x)\right)   ,\quad x\to +\infty.$$
	
Therefore, a Taylor expansion as $x\to +\infty$ yields the asymptotic representation \begin{align*}\ell_{j,y}(x)&=\kappa_j(y)  \left(1+ O\left(   A_{j,y}(x)\right)   \right),\quad A_{j,y} \in {\rm{RV}}_0(+\infty), \quad \kappa_j(y)\in (0,+\infty).
	\end{align*}
\end{proof}

\begin{proof}[Proof of Lemma \ref{lem:speed_beta}]
	Let $y\in E$ and $2\le k \le d$. First, observe that straight consequences of \eqref{eq:quantile} as well as $\Delta$-method  applied to \eqref{eq:gamma} are
	\begin{align}\label{eq:Hill_conv3}
		\frac{\hat {q}_{k,n,y}(\al_n)}{ {q}_{k,y}(\al_n)}=1+ O_{\p} \left(   \tilde\delta^{-1}_{0,n,y}\right)&,\quad	\hat \gamma^{-1}_{n,y}-\gamma^{-1}_y=O_{\p} \left( \tilde\delta^{-1}_{0,n,y} \right) .
	\end{align}
	
	One may write
	\begin{align*}\ln \left( \frac{ \hat{q}_{k,n,y}({\al}_n)}{ \hat{q}_{1,n,y}({\al}_n)} \right)&=\ln \left( \frac{ \hat{q}_{k,n,y}({\al}_n)}{ q_{k,y}(\al_n) } \right)+\ln \left( \frac{ q_{k,y}({\al}_n)}{q_{1,y}({\al}_n)} \right)+\ln \left( \frac{ q_{1,y}({\al}_n)}{ \hat{q}_{1,n,y}({\al}_n)} \right).
	\end{align*}
	
	In view of \eqref{eq:Hill_conv3} and after Taylor expanding, this becomes
	\begin{align}\label{eq:aux1}\ln \left( \frac{ \hat{q}_{k,n,y}(\al_n)}{ \hat{q}_{1,n,y}(\al_n)} \right)&=  \ln \left( \frac{ q_{k,y}(\al_n)}{q_{1,y}(\al_n)} \right)+O_{\p} \left(\tilde\delta^{-1}_{0,n,y} \right) .
	\end{align}

	We investigate the logarithm term in the RHS of \eqref{eq:aux1} and precise its asymptotic behaviour. Regarding the quantiles, it is clear that, for $\al_n \ll 1$ and each $1\le j\le d$, we have $q_{j,y}(\al_n) \gg 1$ and $q_{j,y}\in {\rm{RV}}_{\gamma_y}(1)$. As a functional version of Lemma 3.2 in \cite{MaumeDeschampsRulliereSaidExtremes} which may be proved along the same steps, we have that, if $t=O(s)$, for any $1\le i,j\le d$,
\begin{align*}
	\f{\overline{F}_{i,y}(s)}{\overline{F}_{j,y}(t)} &\sim \f{c_{i,y}}{c_{j,y}}  \left(\f{s}{t}\right)^{-1/\gamma_y},\quad t\to +\infty.
\end{align*} It follows that  \begin{align*}
		1=	\f{\overline{F}_{k,y}(q_{k,y}(\al_n))}{\overline{F}_{1,y}(q_{1,y}(\al_n))} &= c_{k,y} \left(\f{q_{1,y}(\al_n)}{q_{k,y}(\al_n)}\right)^{1/\gamma_y}+o(1),\quad n\to +\infty.
	\end{align*}
	
	Hence,$$  c_{k,y}=\left( \f{q_{k,y}(\al_n)}{q_{1,y}(\al_n)}\right)^{1/\gamma_y}+o(1),\quad \text{i.e.}\quad \f{q_{k,y}(\al_n)}{q_{1,y}(\al_n)}= c_{k,y}^{\gamma_y}+o(1).$$
	
	On the other hand, one may write the ratio in terms of slowly varying functions \begin{align*}
		\f{q_{k,y}(\al_n)}{q_{1,y}(\al_n)}&= \f{U_{k,y}(\overline{\al}_n^{-1})}{U_{1,y}(\overline{\al}_n^{-1})}
		= \f{\ell_{k,y}(\overline{\al}_n^{-1})}{\ell_{1,y}(\overline{\al}_n^{-1})} .
	\end{align*}
	
	By Lemma \ref{lem:slowly_asymptotic_rho}, for $1\le j\neq k \le d$ and as $x\to +\infty$, the following asymptotic holds after a Taylor expansion \begin{align*}\f{\ell_{j,y}(x)}{\ell_{k,y}(x)}&=\f{\kappa_j(y)}{\kappa_k(y)} \left[1+O\left(\max\left( A_{j,y}(x), A_{k,y}(x)\right) \right) \right].
	\end{align*}
	
	We deduce the tail ratio representation $c_{k,y} =\left(  {\kappa_k(y)}/{\kappa_1(y)} \right)^{1/\gamma_y} $ and that for $1\le k \le d$,
	\begin{align}\label{eq:aux2}
		0&\le \limsup\limits_{n\to +\infty} \min( \tilde\delta_{1,n,y}{,}\, \tilde\delta_{k,n,y}) \left(\f{ q_{k,y}(\al_n)}{q_{1,y}(\al_n)}-c_{k,y}^{\gamma_y}  \right) <+\infty.
	\end{align}
	
	We can state \eqref{eq:aux2} in an alternative way $ c_{k,y}^{\gamma_y}  \f{ q_{1,y}(\al_n)}{q_{k,y}(\al_n)} =1+  O\left( \max\left(\tilde\delta^{-1}_{1,n,y},\tilde\delta^{-1}_{k,n,y}\right) \right) ,\quad n\to +\infty.$ 	Combining \eqref{eq:gamma}, \eqref{eq:Hill_conv3}, \eqref{eq:aux1} and \eqref{eq:aux2} with the definition of $\hat c_{k,y}$ in \eqref{eq:hatc_j}, it yields \begin{align*}\ln(\hat{c}_{j,n,y})&=\ff{\hat \gamma_{n,y}}\ln \left( \frac{ \hat{q}_{k,n,y}(\al_n)}{ \hat{q}_{1,n,y}(\al_n)} \right)=  \ff{\hat \gamma_{n,y}}\ln \left( \frac{ q_{k,y}(\al_n)}{q_{1,y}(\al_n)} \right)+O_{\p} \left(\tilde\delta^{-1}_{0,n,y} \right)
		= \f{\gamma_y }{\hat \gamma_{n,y}}   \ln  c_{k,y} +o_\p\left(\max\left(\tilde\delta^{-1}_{1,n,y},\tilde\delta^{-1}_{k,n,y} \right)\right)  +O_{\p} \left(\tilde\delta^{-1}_{0,n,y} \right) .
	\end{align*}
	Limit in \eqref{eq:rate_c} is proved after an application of the $\Delta$-method.  We now consider the quantity
 \begin{align}\label{eq:quantity}\frac{ \beta^{1/\hat{\gamma}_{n,y}}_k}{\hat{c}_{k,n,y}} \left(\frac{\beta^{1/ \gamma_y}_k}{ c_{k,y}} \right)^{-1}&= \exp \left[ - \hat{\gamma}^{-1}_{n,y}  \ln \left( \frac{ \hat{q}_{k,n,y}({\al}_n)}{ \hat{q}_{1,n,y}({\al}_n)} \right)+ \left( \hat \gamma^{-1}_{n,y}-\gamma^{-1}_y \right)   \ln \left( \beta_k \right)+\ln(c_{k,y})\right].
	\end{align}
	
	Combining again \eqref{eq:gamma}, \eqref{eq:Hill_conv3}, \eqref{eq:aux1} and \eqref{eq:aux2}, and Taylor expanding in \eqref{eq:quantity}, it yields\begin{align*}\frac{ \beta^{1/\hat{\gamma}_{n,y}}_k}{\hat{c}_{k,n,y}}-\frac{\beta^{1/ \gamma_y}_k}{ c_{k,y}} &= \f{c_{k,y}^{\gamma_y} }{\gamma_y} O_\p\left(\max \left(\tilde\delta^{-1}_{1,n,y}, \tilde\delta^{-1}_{k,n,y}\right) \right) + O_{\p} \left(\tilde\delta^{-1}_{0,n,y} \right)  .
	\end{align*}
	
	The previous asymptotic equality is valid for any $1\le k \le d$ . Moreover, the asymptotic holds uniformly over $\beta_k \in [\eps_k,M_k]$ with $0<\eps_k<M_k<+\infty$. Indeed, in view of the term $\ln (\beta_k)$ in \eqref{eq:quantity}, we discriminate two cases: when $0<\beta_k\neq 1$ and when $\beta_k=1$.
	
	In the former case, it is clear that $\zp[b]{\ln \beta_k} $ is positive and bounded by a constant depending on $\eps_k,M_k$.
	
	In the second case, the speed rate $\tilde \delta_{0,n,y}$ induced by the difference $\hat \gamma^{-1}_{n,y}-\gamma^{-1}_y$ apparently vanishes but is reintroduced by the term $ \ln ( \frac{ \hat{q}_{1,n,y}({\al}_n)}{ \hat{q}_{k,n,y}({\al}_n)} )$. Now, the term $\eta (1+\sum_{j\neq k} {\beta_j}/{\beta_k} )$ is positive and bounded when $\bm \Theta=(\eta,\beta_2,\ldots,\beta_d)\in \prod_{k=1}^d  [\eps_k,M_k]$ and $\beta_1=1$. Hence, the second part of  Lemma \ref{lem:speed_beta} is proved. 
\end{proof}

\begin{proof}[Proof of Proposition \ref{lem:speed_integral}]
In the following, for the sake of readability,  we drop the finite sum in \eqref{eq:delta3}. 
We start by decomposing $	\Delta^{(3)}_{j,k,n,y} $ defined in \eqref{eq:delta3} as $  	\Delta^{(3)}_{j,k,n,y}  =\Delta^{(3,1)}_{j,k,n,y}+\Delta^{(3,2)}_{j,k,n,y}$ with 
\begin{align*}
\Delta^{(3,1)}_{j,k,n,y}&=\int_{\frac{\beta_j}{\beta_k}}^{+\infty} \lam_{ j,k,y}\left(\frac{\hat c_{j,n,y}}{\hat c_{k,n,y}}t^{-1/{\hat \gamma_{n,y}}},1\right) \mathrm{d}t-\int_{\frac{\beta_j}{\beta_k}}^{+\infty}\lam_{ j,k,y}\left(\frac{c_{j,y}}{c_{k,y}}t^{-1/{\gamma_y}},1\right) \mathrm{d}t,
\end{align*}
\begin{align*}
	\Delta^{(3,2)}_{j,k,n,y}&=\int_{\frac{\beta_j}{\beta_k}}^{+\infty}\hat \lam_{ j,k,n,y}\left(\frac{\hat c_{j,n,y}}{\hat c_{k,n,y}}t^{-1/{\hat \gamma_{n,y}}},1\right) \mathrm{d}t-\int_{\frac{\beta_j}{\beta_k}}^{+\infty} \lam_{ j,k,y}\left(\frac{\hat c_{j,n,y}}{\hat c_{k,n,y}}t^{-1/{\hat \gamma_{n,y}}},1\right) \mathrm{d}t,
\end{align*}
where $\hat \gamma_{n,y}$ and $\hat c_{j,n,y}$, $1\le j \le d$, are defined in \eqref{eq:Hill2} and \eqref{eq:hatc_j}. Note that $\lam_{j,k,y}(0,1)=1-L_{j,k,y}(0,1)=0$. 

In view of Assumption \ref{hyp:integrability}, the mapping $t\mapsto  t^{-(\gamma_y+1)}\lam_{ j,k,y}\left(t,1\right)$ is integrable in the proximity  of $0$ so we may consider the change of variable \begin{align*}
	\frac{c_{j,y}}{c_{k,y}}t^{-1/{\gamma_y}}& = u \iff t =  \left( \frac{c_{k,y}}{c_{j,y}}u\right)^{-\gamma_y},~   \mathrm{d}t =-\gamma_y \left( \frac{c_{k,y}}{c_{j,y}}\right)^{-\gamma_y} u^{-(\gamma_y+1)} \mathrm{d}u.
\end{align*}

It yields the alternative form
\begin{align*}
	\Delta^{(3,1)}_{j,k,n,y}&= -v_y\int_0^{M_y(\bm \Theta)} \left\{\lam_{ j,k,y}\left( \hat s_{j,k,n,y}(t),1\right)-\lam_{ j,k,y}\left(t,1\right) \right\}\mathrm{d}t,
\quad
	\Delta^{(3,2)}_{j,k,n,y}= -\hat v_{n,y}\int_0^{\hat M_{n,y}(\bm \Theta )} \left\{\hat \lam_{ j,k,n,y}\left( t,1\right)-\lam_{ j,k,y}\left(t,1\right) \right\}\mathrm{d}t,
\end{align*}
with \begin{align}\label{eq:hats_n}
\hat s_{j,k,n,y}(t)&:=\f{\hat c_{j,n,y}}{\hat c_{k,n,y}}\left(\f{c_{k,y}}{c_{j,y}}\right)^{\gamma_y/\hat\gamma_{n,y}}\hspace{-0.5cm}\cdot t^{\gamma_y/\hat\gamma_{n,y}}:=\hat \sigma_{n,y}   \cdot\, t^{\gamma_y/\hat\gamma_{n,y}},
\end{align}
$$M_{y}(\bm \Theta ):=  \frac{c_{j,y}}{c_{k,y}}\left(  \frac{\beta_j}{\beta_k}\right)^{-1/{\gamma_y}},\quad v_{y} := \gamma_y \left( \frac{c_{k,y}}{c_{j,y}}\right)^{-\gamma_y},\quad
\hat M_{n,y}(\bm \Theta ):=  \frac{\hat c_{j,n,y}}{\hat c_{k,n,y}}\left(  \frac{\beta_j}{\beta_k}\right)^{-1/{\hat \gamma_{n,y}}},\quad \hat v_{n,y} := \hat \gamma_{n,y} \left( \frac{\hat c_{k,n,y}}{\hat c_{j,n,y}}\right)^{-\hat \gamma_{n,y}}\hspace{-0.5cm}.$$

In order to prove Proposition \ref{lem:speed_integral}, it is enough to show that, for $s\in \{1,2\}$, \begin{align}\label{eq:boil2}
\delta_{n,y}\zp[b]{	\Delta^{(3,s)}_{j,k,n,y}(\bm \Theta)}\xrightarrow[n\to +\infty]{\p}0,
\end{align}
and the uniform counterparts, for $0<\eps_r<M_r<+\infty$, $1\le r\le d$,
$	\delta_{n,y} \sup_{\bm \Theta \in \prod_{p=1}^d [\eps_r,M_r]  } \zp[b]{	\Delta^{(3,s)}_{j,k,n,y}(\bm \Theta)}\xrightarrow[n\to +\infty]{\p}0.
$

From \eqref{eq:gamma} and \eqref{eq:rate_c}, for any $y\in E$, $1\le j\le d$ and $t\ge 0$, we have $
	\hat \gamma_{n,y}  \xrightarrow[n\to +\infty]{\p} \gamma_y,\quad \hat c_{j,n,y}\xrightarrow[n\to +\infty]{\p} c_{j,y}.$

In particular, the continuous mapping theorem readily implies that $\hat s_{j,k,n,y}(t)\xrightarrow[n\to +\infty]{\p} t$ and also \begin{align}\label{eq:bounded}
  \hat M_{n,y}(\bm \Theta )\xrightarrow[n\to +\infty]{\p} M_y(\bm \Theta ),\quad \hat v_{n,y}(t)\xrightarrow[n\to +\infty]{\p} v_y.
\end{align}

It is classic knowledge that the stable tail dependence function $\lam_{j,k,y}$ is $1$-Lipschitz, hence, the case $\Delta^{(1)}_{n,y}$ in \eqref{eq:boil2} follows from Lemma \ref{lem:s_n-t}. On the other hand, the case $\Delta^{(2)}_{k,n,y}$ in \eqref{eq:boil2} is a mere consequence of the uniform boundedness over $\beta := (1,\beta_2,\ldots,\beta_d)\in\prod_{j=1}^{d}[\eps_j,M_j]$ with $0<\eps_j<M_j<+\infty$, for any $j$, of the quantities in \eqref{eq:bounded} and the uniform convergence in probability of Lemma \ref{lem:lam_conv}.
\end{proof}

\subsection{Auxiliary lemmas}

In this section we introduce and prove  three  auxiliary results. Recall the scaling sequence $\delta_{k,n,y}$ for each margins $1\le k \le d$ defined in \eqref{eq:scaling2}. Firstly we may have a closer look to the previous convergence of $\hat s_{j,k,n,y}$ in \eqref{eq:hats_n}. It follows from  \eqref{eq:gamma} and \eqref{eq:rate_c} that $$ \delta_{k,n,y}  \left( \hat \sigma_{n,y}-1\right) \xrightarrow[n\to +\infty]{\p}0. $$

An application of the multivariate $\Delta$-method yields the following pointwise convergence in probability
\begin{align*}
\delta_{k,n,y}  \left( \hat s_{j,k,n,y}(t)-t \right)   &\xrightarrow[n\to +\infty]{\p} 0.
\end{align*}

Leaning on the monotonocity of the mapping $t\mapsto \hat s_{j,k,n,y}(t)$, we improve the pointwise convergence in probability $\hat s_{j,k,n,y}(t)-t=o_{\p}(1)$ to an uniform version by a probabilistic Dini-type argument.

\begin{lemma}\label{lem:dini} For any $0<T<+\infty$ and $1\le j\neq k\le d$,  it holds that 
	$  \sup_{t\in [0,T]}\zp[b]{ \hat s_{j,k,n,y}(t)-t}\xrightarrow[{n\to +\infty}]{\p} 0.$
\end{lemma}

\begin{proof}[Proof of Lemma \ref{lem:dini}]
Let $T>0$ and $\eps>0$. Let also an integer $N:=N(\eps)\ge 1$ such that $T<\eps N$. We consider a subdivision with step less than $\eps$ of the compact set $[0,T]=\cup_{\ell=0}^{m-1} [\f{\ell}{N}T,\f{\ell+1}{N}T]$. In particular, for any $t\in [0,T]$, there exists $0\le \ell_1<\ell_2\le N$ such that $\f{\ell_1}{N}T\le t\le \f{\ell_2}{N}T$. Fix $\omega\in \Omega$ for the moment and assume that for any $0\le \ell\le N$, $\zp[b]{\hat s_{j,k,n,y}\left(\f{\ell}{N}T\right)(\omega)-\f{\ell}{N}T}<\eps$. Since the mapping $x\mapsto \hat s_{j,k,n,y}(x)(\omega)$ is non-decreasing, the following inequalities hold $$\hat s_{j,k,n,y}\left(t\right)(\omega)-t\le \hat s_{j,k,n,y}\left(\f{\ell_2}{N}T\right)(\omega)-\f{\ell_1}{N}T \le  s_{j,k,n,y}\left(\f{\ell_2}{N}T\right)(\omega)-\f{\ell_2}{N}T+\eps < 2\eps ,  $$
$$\hat s_{j,k,n,y}\left(t\right)(\omega)-t\ge \hat s_{j,k,n,y}\left(\f{\ell_1}{N}T\right)(\omega)-\f{\ell_2}{N}T \ge  \hat s_{j,k,n,y}\left(\f{\ell_1}{N}T\right)(\omega)-\f{\ell_1}{N}T-\eps > -2\eps .  $$ Thus, $\sup_{t\in [0,T]} \zp[b]{\hat s_{j,k,n,y}\left(t\right)(\omega)-t}<2\eps$. As a consequence, $$\p \left(  \zp[b]{\hat s_{j,k,n,y}\left(\f{\ell}{N}T\right)-\f{\ell}{N}T}<\eps,~\forall \ell=0,\ldots, N\right)\le \p\left(\sup_{t\in [0,T]} \zp[b]{\hat s_{j,k,n,y}\left(t\right)-t}<2\eps\right). $$

Taking the complementary, using union bound and the previous  pointwise convergence in probability of $\hat s_{j,k,n,y}$, as $N,T,\eps<+\infty$ are independent of $n$,
$$\p\left(\sup_{t\in [0,T]} \zp[b]{\hat s_{j,k,n,y}\left(t\right)-t}>\eps\right)\le \sum_{\ell=0}^{N} \p \left(  \zp[b]{\hat s_{j,k,n,y}\left(\f{\ell}{N}T\right)-\f{\ell}{N}T}>\f{\eps}{2}\right)\xrightarrow[n\to +\infty]{}0. $$
\end{proof}

In the following lemma we provide  the speed rate of  convergence in Lemma \ref{lem:dini}.

\begin{lem}\label{lem:s_n-t}
Let   $\delta_{k,n,y}$ be the  scaling sequence  for each margins $1\le k \le d$ defined in \eqref{eq:scaling2}.  Let $y\in E$ and $0<T<+\infty$.  Under Assumptions  \ref{hyp:model1}, \ref{hyp:model2} and \ref{hyp:regime}, \ref{hyp:oscillation}, \ref{hyp:regime2}, \ref{hyp:kernel}, for $j\neq k$,
$$ \delta_{k,n,y}	\sup_{t\in [0,T]}\zp[b]{\hat s_{j,k,n,y}(t)-t} \xrightarrow[n\to +\infty]{\p}0.$$
\end{lem}
\begin{proof}[Proof of Lemma \ref{lem:s_n-t}]
Recall  $\hat \sigma_{n,y}:=\f{\hat c_{j,n,y}}{\hat c_{k,n,y}}\left(\f{c_{k,y}}{c_{j,y}}\right)^{\gamma_y/\hat\gamma_{n,y}}$\hspace{-0.2cm}. 
Define $g_{j,k,n,y}(t):=\hat s_{j,k,n,y}(t)-t$ for $t\ge 0$ with $\hat s_{j,k,n,y}$ as in \eqref{eq:hats_n}. Consider $t_{n,y}:= \left( \f{\gamma_y}{\hat\gamma_{n,y}} \hat \sigma_{n,y}\right)^{ \f{\hat\gamma_{n,y}}{\hat\gamma_{n,y}-\gamma_y}} $ the solution of $$ g'_{j,k,n,y}(t) = \hat \sigma_{n,y}\f{\gamma_y}{\hat\gamma_{n,y}}t^{\gamma_y/\hat\gamma_{n,y}-1}-1=0,\quad t\ge 0.$$  This corresponds to the maximum value of $g_{j,k,n,y}$ when $\gamma_y <\hat \gamma_{n,y}$ or its minimum value when $\gamma_y>\hat \gamma_{n,y}$. This is a consequence of the 2nd order optimality condition and $$g''_{j,k,n,y}(t) =  \hat \sigma_{n,y}\f{\gamma_y}{\hat\gamma_{n,y}}\left( \gamma_y/\hat\gamma_{n,y}-1\right) t^{\gamma_y/\hat\gamma_{n,y}-2},\quad t\ge 0. $$

Using \eqref{eq:gamma} and \eqref{eq:rate_c}, we easily see that \begin{align*}
	t_{n,y}&=	\left( \f{\gamma_y}{\hat\gamma_{n,y}} \hat \sigma_{n,y}\right)^{ \f{\hat\gamma_{n,y}}{\hat\gamma_{n,y}-\gamma_y}}=\exp\left( \f{\hat\gamma_{n,y}}{\hat\gamma_{n,y}-\gamma_y} \ln \left( \f{\gamma_y}{\hat\gamma_{n,y}} \hat \sigma_{n,y} \right)  \right)
	= \exp\left( \f{\hat\gamma_{n,y}}{\hat\gamma_{n,y}-\gamma_y} o_{\p}\left(\delta_{k,n,y}^{-1}\right)  \right) = 1+o_{\p}\left(1\right)  .
\end{align*}

For fixed $\omega\in \Omega$, $0<T<+\infty$ and $n$ large enough, we may write
\begin{align*}
	\sup_{t\in [0,T]}\zp[b]{g_{j,k,n,y}(t)}(\omega) &\le \max \left( \zp[b]{g_{j,k,n,y}(t_{n,y})}(\omega) , \zp[b]{g_{j,k,n,y}(T)}(\omega)\right).
\end{align*}

By Taylor expansion and multivariate $\Delta$-method with \eqref{eq:gamma} and \eqref{eq:rate_c}, one readily shows that for $0<T<+\infty$ fixed, $g_{j,k,n,y}(T) =o_{\p}( 1/ \delta_{k,n,y}) $. Now, we compute \begin{align*}g_{j,k,n,y}(t_{n,y})&=   
	  t_{n,y} \left( \f{\hat\gamma_{n,y} }{\gamma_y  }- 1\right).\end{align*}
   

Then, using \eqref{eq:gamma} and \eqref{eq:rate_c} once more yields $g_{j,k,n,y}(t_{n,y} )=o_{\p}( 1/ \delta_{k,n,y}) $.
\end{proof}
  
\begin{lem}\label{lem:lam_conv}
	{Let $M\in (0,+\infty)$. Let $1\le j\ne k\le d$,  $y\in E$   and  $1\ll k_n\ll n$ be an intermediate sequence. Under Assumptions \eqref{hyp:segers}, \eqref{hyp:rate_theoretical}, \ref{hyp:tau}, \ref{hyp:continuity}, \ref{hyp:beta}, \ref{hyp:kernel}, we define 
 \begin{center}$1\ll \delta_{0,n,y}\ll \tilde\delta_{0,n,y} = \sqrt{n\overline{\al}_n\psi_y(h)}$ \quad   and \quad   $1\ll \delta_{-1,n,y}\ll \tilde{\delta}_{-1,n,y}=\left(\f{n}{k_n}\right)^{\,\mu_y}$ \end{center} as in 
 Equation \eqref{eq:scaling1}   with $\mu_y>0$ as in Condition  \eqref{hyp:rate_theoretical}.   }
 Then, it holds that	\begin{align*} \min\left( \delta_{0,n,y},\delta_{-1,n,y}\right)  \sup_{\bm x \in [0,M]^2}\zp[b]{\lam_{j,k,y}(\bm x)-\hat \lam_{j,k,n,y}(\bm x) }& \xrightarrow[n\to +\infty]{\p}0.
	\end{align*}
\end{lem}

\begin{proof}[Proof of Lemma \ref{lem:lam_conv}]
	By using Equation  \eqref{eq:lam_def_cond}, we write  \begin{align*}\sup_{\bm x \in [0,M]^2}\zp[b]{\lam_{j,k,y}(\bm x)-\hat \lam_{j,k,n,y}(\bm x) }=\sup_{\bm x \in [0,M]^2}\zp[b]{L_{j,k,y}(\bm x)-\hat L_{j,k,n,y}(\bm x) }&.
	\end{align*}
{For  $n$ large enough, one can get $\frac{k_n}{n}M \le 1$.  	Then, the }triangular inequality allows us to bound the latter quantity by
	\begin{align*}& \sup_{\bm x \in [0,M]^2}\zp[b]{L_{j,k,y}(\bm x)-\tilde L_{j,k,n,y}(\bm x)}+\sup_{\bm x \in [0,M]^2}\zp[b]{\hat L_{j,k,n,y}(\bm x)-\tilde L_{j,k,n,y}(\bm x)},
	\end{align*} where $\tilde L_{j,k,n,y}(\bm x):=\f{n}{k_n}\left( 1- C_{j,k,y} (\bm 1- \frac{k_n}{n}\bm x) \right)$. Concerning the second term, since $\overline{\al}_n = o(1)$, it readily follows from \eqref{eq:irene} that for any deterministic $1\ll  \delta_{0,n,y}\ll \tilde\delta_{0,n,y}\ll \sqrt{n\psi_y(h_n)} $,
	\begin{align*} \delta_{0,n,y} \sup_{\bm x \in [0,M]^2} \zp[b]{\hat L_{j,k,n,y}(\bm x)  -\tilde L_{j,k,n,y}(\bm x) } &\xrightarrow[n\to +\infty]{\p}0.
	\end{align*}
	Lastly, using the rate condition  in \eqref{hyp:rate_theoretical}, we have  $\delta_{-1,n,y} \sup_{\bm x \in [0,M]^2}\zp[b]{ L_{j,k,y}(\bm x)-\tilde L_{j,k,n,y}(\bm x)}\xrightarrow[n\to +\infty]{\p}0. $ 
\end{proof}

 \section*{Conclusion}
In this paper, we present a semi-parametric method for estimating functional multivariate $L^1$-expectiles at extreme risk levels (i.e., functional $L^1$-MEEs). Going beyond \cite{BeckMailhotElena2021}, we establish, in this functional setting, the consistency   with rate of the approximated loss function by using empirical kernel-based estimators for the  tail index, the  tail ratio and the  upper tail dependence function. Then, we couple the loss function estimation with a gradient descent algorithm (e.g., BFGS-family). As a result, we give the consistency with rate of the approximated optimum problem for solving functional $L^1$-MEEs.  
\\
An immediate line of work is the implementation of the MEEs estimation method and testing its finite-sample performance. Adopting the BFGS family method as gradient descent algorithm, one may proceed to the simulation study for Lomax marginals and survival Clayton as dependence structure. Such model is closely related to \cite{GSU2022,GG2012,BeckMailhotElena2021} but taking into account that Pareto I is not second order regularly varying.
\\
Furthermore, one may raise the question of how to calibrate the parameters (especially from the kernel smoothing). A cross-validation procedure such as \cite{GSU2022} is in consideration for the choice of the bandwidth  associated to the small ball probability hypothesis of the considered functional space and the risk level of the MEEs.
\\
Another possible line of research would be to consider the $L^p$-norm in Equation \eqref{def:cond_multivariate_L1-expectile1}, i.e., to consider the  $L^p$-expectiles defined in \cite{maumedeschamps:extension} in this functional setting. The question of their estimation may also be approached by the same semi-parametric methodology although it would display more intricate algebra. Along the same lines, one may consider a possible extension of \cite{MaumeDeschampsRulliereSaidExtremes} to functional extreme $\Sigma$-expectiles for $\Sigma \ne \bm 1$.
\\
While the specific case of one dominant marginal is studied in \cite{MaumeDeschampsRulliereSaidExtremes}, an interesting problem is to remove the  multivariate regular variation hypothesis on the random vector $\bm X$ with  equivalent marginals tails. 
\\
Finally, a more ambitious future work would be to connect our inference method with the online optimization theory. Indeed, we propose in our work to simultaneously estimate and optimize the loss function associated to the considered  functional $L^1$-MEEs, which may essentially be seen as updating with new data-points  a multidimensional optimization problem.

\subsection*{Acknowledgement}
This work has been supported by the project ANR McLaren (ANR-20-CE23-0011). This work has been partially supported by the French government, through the 3IA Côte d’Azur Investments in the Future project managed by the  National Research Agency (ANR) with the reference number ANR-19-P3IA-0002.

\bibliographystyle{plain}
\bibliography{biblio_expectiles}

\begin{appendix} 

\section{Regular variation framework}\label{sec:RVF} 
\begin{Def}[Multivariate Regular Variation]\label{def:MRV}
		A random vector $\bm X\in \R^d$ is regularly varying with index $\alpha\ge 0$ if  there exists a random vector $\Theta$ on the unit sphere $\mathbb{S}^{d-1}$ such that for any $x\in (0,+\infty)$, the following vague convergence holds:    \begin{align*}\f{
				\p\left( \zp[n]{\bm X} > x\,z, \f{\bm X}{\zp[n]{\bm X}}\in \cdot \right)}{\p\left(  \zp[n]{\bm X} > z\right) }	&\xrightarrow[z\to +\infty]{} x^{-\al} \p\left(\Theta \in \cdot \right).
		\end{align*}
	\end{Def}

\begin{Def}[Second order regular variation \cite{extreme_value_theory_de_haan_ferreira}]\label{def:soc1}
	Let $\gamma >0$ and $\rho \le 0$. We say that $U \in 2{\rm{RV}}_{-\gamma,\rho}(+\infty)$ if there exists an auxiliary function $A$ with ultimately constant sign and $\lim\limits_{t\to +\infty}A(t)=0$ such that \begin{align}\label{eq:soc1}
		\lim\limits_{t\to +\infty}	\ff{A(t)}\left( \f{U(tx)}{U(t)}- x^{\gamma} \right)&=x^{\gamma} H_\rho(x):=x^{\gamma}\int_1^x u^{\,\rho-1}\mathrm{d}u,\quad  x>0.
	\end{align}
\end{Def}
Note that for $x>0$,  $$H_\rho(x)=\ln(x)\bm{1}_{\{\rho=0\}} +\f{x^{\,\rho}-1}{\rho}\bm{1}_{\{\rho<0\}} .$$

By Theorem 2.3.9 in \cite{extreme_value_theory_de_haan_ferreira}, this is equivalent to
\begin{align}\label{eq:soc12}	\lim\limits_{t\to +\infty}	\ff{A\left( 1/{\overline{F}(t)} \right) }\left( \f{\overline{F}(tx)}{\overline{F}(t)}- x^{-\ff{\gamma} } \right)&= H_{\rho,\gamma}(x):= x^{-\ff{\gamma}}  \f{x^{\f{\rho}{\gamma}}-1}{\gamma\rho },\quad x>0.
\end{align}

Namely,$$U\in 2{\rm{RV}}_{\gamma,\rho}(+\infty)\iff  \overline{F} \in 2{\rm{RV}}_{-\ff{\gamma},\f{\rho}{\gamma}}(+\infty).$$

Besides, on the auxiliary functions level, Theorem 2.3.3 in \cite{extreme_value_theory_de_haan_ferreira} provides $$A\in {\rm{RV}}_{\rho}(+\infty),\quad A \circ  \ff{\overline{F}} \in {\rm{RV}}_{\frac{\rho}{\gamma}}(+\infty).$$

\begin{rmk} 
Condition \ref{hyp:model1} entails a second order regular variation behaviour. Note that $c_{1,y}\equiv 1$ for any $y\in E$. For any $1\le j \le d$, since $\overline{F}_{j,y} \in {\rm{RV}}_{-1/\gamma_y}(+\infty)$, there exists $\ell_{\overline{F}_{j,y}}\in {\rm{RV}}_0(+\infty)$, such that \begin{align}\label{eq:L_representation}\overline{F}_{j,y}(x)&=x^{-1/{\gamma_y}}\ell_{\overline{F}_{j,y}}(x),\quad x>0.
\end{align}

Also, we have $U_{j,y}\in {\rm{RV}}_{\gamma_y}(+\infty)$ and thus, there also exists a slowly varying function $\ell_{U_{j,y}}:=\ell_{j,y}$ such that ${U_{j,y}}(x)=x^{\gamma_y} \ell_{j,y}(x)$, for $ x> 0$.
 Condition \ref{hyp:model2} is inherited from \cite{BeckMailhotElena2021,MaumeDeschampsRulliereSaidExtremes} and states that each marginal behaves the same way in the extreme regime.
Moreover, since $ \overline{F}_{j,y} \in 2{\rm{RV}}_{-1/{\gamma_y},{\rho_{j,y}}/{\gamma_y}}(+\infty)$, it follows that $U_{j,y}\in 2{\rm{RV}}_{\gamma_y,\rho_{j,y}}(+\infty)$. According to Definition \ref{def:soc1}, there exists auxiliary functions $A_{j,y}:=A_{U_{j,y}}\in {\rm{RV}}_{\rho_{j,y}}(+\infty)$ satisfying Equation \eqref{eq:soc1}.
\end{rmk}

\begin{rmk}\label{exampleFbar} Examples of distributions satisfying the second order regular variation condition are:
	\begin{itemize}\item Log-gamma distribution defined as the exponential of the sum of two independent standard exponential random variables: $\overline{F}(x)=\ff{x}\left( 1+\log(x)\right) \in 2{\rm{RV}}_{-1,0}(+\infty)$ with $A( 1/{\overline{F}(t)} )=1/\log(t)$ (see \cite{soc}).
        \item  For  $x>e$, \, $\overline{F}(x)=\f{\log x}{x}e \,\in \, {\rm 2RV}_{-1,0}(+\infty),\,$ with $A( 1/{\overline{F}(t)} )=1/\log(t)$.
		\item Hall-Weiss class: $\overline{F}(x)=\ff{2} x^{-\al} (1+x^\rho)\in 2{\rm{RV}}_{-\al,\rho}(+\infty)$ for $\alpha>0$ and $\rho <0$ with $A( 1/{\overline{F}(t)} )=\rho t^\rho$ (see \cite{soc}).
		\item Cauchy distribution: $\overline{F}(x) = \ff{\pi}\tan^{-1}(1/x)+\ff{2}(1-{\rm{sign}}(x)) \in 2{\rm{RV}}_{-1,-2}(+\infty) $ (see \cite{GG2012}).
		\item Fréchet distribution: $\overline{F}(x)=1-e^{-x^{-\ff{\gamma}}} \in 2{\rm{RV}}_{-\gamma,-\gamma}(+\infty)$, $\gamma>0$ (see \cite{GG2012}).
		\item Burr distribution: $\overline{F}(x)=\left(1+x^{\tau}\right)^{-\lam} \in 2{\rm{RV}}_{-\ff{\tau\lam},-\tau}(+\infty)$, $\tau,\lam>0$ (see \cite{GG2012}).
		\item Student $t_\nu$, $\nu>1$, distribution with density $f(x)=\f{\Gamma\left( \f{\nu+1}{2}\right) }{\sqrt{\nu \pi}\Gamma\left( \f{\nu}{2}\right)}\left(1+\f{x^2}{\nu} \right)^{-\f{\nu+1}{2}}$. Then, $\overline{F}\in2{\rm{RV}}_{-\nu,-2}(+\infty)$.
  	\end{itemize} 
The interested reader can find more 2RV type distributions in Section 5 of \cite{Mao}.  
\end{rmk}
\end{appendix}
\end{document}